\documentclass[11pt]{article} 

\usepackage[utf8]{inputenc} 
\usepackage[english]{babel} 
\usepackage{amsmath}
\usepackage{amssymb}
\usepackage{txfonts}
\usepackage{mathdots}
\usepackage[classicReIm]{kpfonts}
\usepackage{graphicx}
\usepackage{enumitem}
\usepackage{xcolor}
\usepackage{pifont}
\usepackage[colorlinks=true,linkcolor=blue]{hyperref}
\usepackage[nameinlink]{cleveref}

\usepackage{url}
\usepackage{cite}
\usepackage[frak=esstix]{mathalpha}

\usepackage{stmaryrd} 

\usepackage{fancyhdr}

\usepackage{mathtools}
\usepackage[left=3cm,right=3cm,top=3cm,bottom=3cm]{geometry}

\newcounter{propertyCounter}
\newtheorem{property}{Property}[propertyCounter]

\newcounter{theoremCounter}
\newtheorem{theorem}{Theorem}[theoremCounter]

\newcounter{corollaryCounter}
\newtheorem{corollary}{Corollary}[corollaryCounter]

\newcounter{propositionCounter}
\newtheorem{proposition}{Proposition}[propositionCounter]

\newcounter{lemmaCounter}
\newtheorem{lemma}{Lemma}[lemmaCounter]

\newcounter{definitionCounter}
\newtheorem{definition}{Definition}[definitionCounter]

\newcounter{noteCounter}
\newtheorem{note}{Note}[noteCounter]

\newcounter{remarkCounter}
\newtheorem{remark}{Remark}[remarkCounter]

\newcounter{notationCounter}
\newtheorem{notation}{Notation}[notationCounter]

\newcounter{conjectureCounter}
\newtheorem{conjecture}{Conjecture}[conjectureCounter]

\newcounter{exampleCounter}
\newtheorem{example}{Example}[exampleCounter]

\usepackage{lastpage}
\title{Theorems Related to Fermat's Sums of Two Squares}
\author{WOLF Fran\c{c}ois, WOLF Marc.\\}
\date{\today}

\begin{document}

\maketitle

\noindent 
\noindent \textbf{WOLF Fran\c{c}ois,  }\url{https://orcid.org/0000-0002-3330-6087}\underbar{} \\
\noindent Email: francois.wolf@dbmail.com\textbf{} \\
\noindent \textbf{WOLF Marc, }\url{https://orcid.org/0000-0002-6518-9882} \underbar{} \\
\noindent Email: marc.wolf3@wanadoo.fr\underbar{}

\begin{abstract}
\noindent In this paper, we study the factorization of sums of two squares $X^2+Y^2$. We show the existence of linear and quadratic progressions on $X$ that generate both factorizations and alternative decompositions into sums of two squares. We are therefore naturally interested in the number of decompositions into sums of two squares of an integer $n$ and relate it to its number of divisors $d\left(n\right)$. Finally, we present an algorithm to obtain these decompositions using the sieve introduced in \cite{a20}.
\end{abstract}

\noindent \textbf{\underbar{Keywords}}: Euler factorization, Fermat composite numbers, quadratic progression, Diophantus’ identity, theorems related to Fermat's sum of two squares, decomposition into a sum of two squares, sieve, prime numbers.\eject 

\noindent 


\noindent

\section{Introduction}

\noindent The question of sums of two squares plays a central role in the history of number theory: ranging from Fermat's observations on the representations of prime numbers $p\mathrm{\equiv }\mathrm{1\ }\left[\mathrm{4}\right]$, through Euler's identities and the structure of Gaussian integers, to modern analytical work on the distribution of these representations (Landau and his successors), this object links algebraic,  arithmetic and analytical tools.
\\[10pt]
\noindent We mention a few references, including classical literature and modern syntheses:

\begin{enumerate}[label=-]
\item   The algebraic approaches (Gaussian integers) studied by C. F. Gauss \cite{b1}, D. A. Cox \cite{b2}, H. Davenport \cite{b3}, G. H. Hardy \& E. M. Wright \cite{b4}, E. Landau \cite{b5},

\item  The arithmetic approaches to Z[i] (two-square theorem and generalizations) studied by K. Ireland \& M. Rosen \cite{b6}, H. Cohn \cite{b7} \cite{b8},

\item  The analytic and modular approaches (modular functions, automorphic forms, distribution of representations) studied by E. Grosswald \cite{b9}, H. Iwaniec \& E. Kowalski \cite{b10}, T. M. Apostol \cite{b11}, H. Iwaniec \cite{b12}, K. Ono \cite{b13}, H. H. Chan, C. Krattenthaler \cite{b14}.
\end{enumerate}

\noindent Understanding not only which integers are written in the form $n\mathrm{=}X^{\mathrm{2}}\mathrm{+}Y^{\mathrm{2}}$ but also how these decompositions relate to factorizations $n\mathrm{=}AB$ is at the heart of the developments presented here.
\\[10pt]
\noindent In their previous work \cite{a20}, \cite{a30}, \cite{a40}, the authors studied prime or composite integers of the form $X^{\mathrm{2}}\mathrm{+}c$. Here we require $c\mathrm{=}Y^{\mathrm{2}}$ to be a perfect square and we explain with elementary tools the correspondence between non-trivial factorizations and representations as the sum of two squares. The objective of this paper is threefold: first, to classify the non-trivial integers of the form $X^{\mathrm{2}}\mathrm{+}Y^{\mathrm{2}}$ according to arithmetic invariants; second, to explicitly construct infinite families of values $X$ (linear or quadratic sequences) such that the integers $X^{\mathrm{2}}\mathrm{+}Y^{\mathrm{2}}$ are composite in a similar manner; third, to apply these constructions – in particular, we show the infinity of certain categories stemming from our classification, and we extend a sieve algorithm to yield all decompositions into sums of two squares of given integers.

\begin{definition}{:}

\begin{enumerate} 
\item  We denote $E_c=\left\{X^2+c,\ X\in 2\mathbb{Z}\mathrm{+}1+c\right\}$ for $c\in {\mathbb{N}}^*$.

\item  The set of divisors of an integer $x$ is denoted $\mathcal{D}\left(x\right)\coloneqq \left\{d\in {\mathbb{N}}^*\mathrel{\left|\vphantom{d\in {\mathbb{N}}^* d|x}\right.\kern-\nulldelimiterspace}d|x\right\}$, and by extension the set of divisors of at least one element of a set $E\subset \mathbb{Z}$ is denoted $\mathcal{D}\left(E\right)\coloneqq \bigcup_{x\in E}{\mathcal{D}\left(x\right)}$. We also define ${\mathcal{D}}_p\left(x\right)=\mathcal{D}\left(x\right)\cap \mathbb{P}$ the set of prime divisors of $x$ and ${\mathcal{D}}_p\left(E\right)=\mathcal{D}\left(E\right)\cap \mathbb{P}$ the set of prime divisors of $E$.

\item  Let $E^*_c=\left\{n\in E_c|{\mathrm{gcd} \left(n,c\right)\ }=1\right\}$. The composite integers of $E^*_c$ are called \underbar{non-trivial}. We denote $E^{\left(T\right)}_c$ the complement of $E^*_c$ in $E_c$, i.e. the set consisting of the \underbar{trivial} composite integers of $E_c$, and of $c$ itself when it is odd and different from 1.

\item  For $\left(x,y\right)\in \mathbb{N}\mathrm{\times }{\mathbb{N}}^*$ with  $x<y$, we denote ${\mathcal{P}}_{x,y}\coloneqq \left\{p\in \mathbb{P}\mathrm{\ }\mathrm{|}\mathrm{\ }p\equiv x\ \left[y\right]\right\}$.

\item  If $n\in {\mathbb{Z}}^*$ and $p\in \mathbb{P}$, we denote ${\nu }_p\left(n\right)$ the $p$-adic valuation of $n~$:
\begin{center}
${\nu }_p\left(n\right)={\mathrm{max} \left\{k\in \mathbb{N}\mathrel{\left|\vphantom{k\in \mathbb{N} p^k|n}\right.\kern-\nulldelimiterspace}p^k|n\right\}\ }.$
\end{center}
\end{enumerate}
\end{definition}

\begin{property}{:}
$E^{\left(T\right)}_c\neq \emptyset $ if and only if ${\mathcal{D}}_p\left(c\right)\diagdown \left\{2\right\}\neq \emptyset $, and in this case the trivial integers $X^2+c$ of $E^{\left(T\right)}_c$ are given by a union of arithmetic progressions \cite{a20}:
\begin{center}
$E^{\left(T\right)}_c=\bigcup_{p\in {\mathcal{D}}_p\left(c\right)\diagdown \left\{2\right\}}{\left\{{\left(px\right)}^2+c\in E_c,\ x\in 2\mathbb{Z}\mathrm{+}1+c\right\}}.$
\end{center}
\end{property}

\noindent 
\section{ Euler factorization.}

\noindent 
\subsection{Rewriting the Factorization of the Sum of Two Squares}

\noindent Any integer $X\in 2{\mathbb{N}}^*$ is uniquely written as $X=2^h\left(2N+1\right)$ with $h={\nu }_2\left(X\right)\in {\mathbb{N}}^*,N\in \mathbb{N}$. Let $Y^2$ be an \underbar{odd} perfect square. In this section, we use Gaussian integers to obtain necessary conditions on the factorizations of the elements in $E^*_{Y^2}$.

\begin{definition}{: Let}
\[\mathcal{Q}:\left(h,N,Y\right)\in {\mathbb{N}}^*\mathrm{\times }\mathbb{N}\mathrm{\times }\left(2\mathbb{N}\mathrm{+}1\right)\mapsto {\left(2^h\left(2N+1\right)\right)}^2+Y^2.\] 
For all $Y\in 2\mathbb{N}\mathrm{+}1$, we thus have $\mathcal{Q}\left({\mathbb{N}}^*\mathrm{,}\mathbb{Z}\mathrm{,}Y\right)=E_{Y^2}$.
\end{definition}

\noindent 
\subsection{Gaussian integers}

\noindent We recall some fundamental results regarding the ring $\mathbb{Z}\left[i\right]$ of complexes with integer coefficients, also known as the ring of \underbar{Gaussian integers}.

\begin{property}{:}
The units (or invertible elements) of $\mathbb{Z}\left[i\right]$ are $1,-1,i,-i$.
\end{property}

\begin{definition}{:}
An element $a+ib\in \mathbb{Z}\left[i\right]$ is said to be \underbar{irreducible} if it is not invertible but cannot be decomposed into the product of two non-invertible elements.
\end{definition}

\begin{example}{:}
$1+i,\ 3,\ 1+2i$ are irreducible. $2=\left(1+i\right)\left(1-i\right)$ and $5=\left(1+2i\right)\left(1-2i\right)$ are not.
\end{example}

\begin{property}{:}
The set of irreducible elements of $\mathbb{Z}\left[i\right]$ are (up to a unit) the elements of ${\mathcal{P}}_{3,4}$ and the complexes $x+iy$ such that $x^2+y^2\in \left\{2\right\}\cup {\mathcal{P}}_{1,4}$.
\end{property}

\begin{remark}{:}
To any irreducible element $z$, there correspond three other associated irreducible elements, namely $iz,-z,-iz$. If $x+iy$ is irreducible, the same holds for $x-iy$. Up to multiplication by a unit, we can therefore assume $x\ge \left|y\right|$. We denote by $Z_{\mathbb{P}}$ the set of irreducible of $\mathbb{Z}\left[i\right]$ of this form, omitting $1-i=-i\left(1+i\right)$ which is already accounted for by $\left(1+i\right)$. Let $Z^+_{\mathbb{P}}$ (resp. $Z^-_{\mathbb{P}}$) the set of elements of $Z_{\mathbb{P}}\diagdown \left\{1+i\right\}$ such that $\mathfrak{I}\left(z\right)>0$ (resp. $\mathfrak{I}\left(z\right)<0$). In summary, we have the partition :
\begin{center}
 $Z_{\mathbb{P}}=\left\{1+i\right\}\cup \left(Z^+_{\mathbb{P}}\cup Z^-_{\mathbb{P}}\right)\cup {\mathcal{P}}_{3,4}$
\end{center}
\end{remark}

\begin{theorem}{:}
\label{t1}
$\mathbb{Z}\left[i\right]$ is a Euclidean domain. For any non-zero Gaussian integer $a+ib\in \mathbb{Z}\left[i\right]\diagdown \left\{0\right\}$, there exists a unique unit $w$ and a unique finitely supported family  ${\left(\alpha_z\right)}_{z\in Z_{\mathbb{P}}}$ such that:
\[a+ib=w\prod_{z\in Z_{\mathbb{P}}}{z^{\alpha_z}}.\]
\end{theorem}

\begin{remark}{:}
In the case where $a+ib=n\in {\mathbb{N}}^*$, we have $w={\left(-i\right)}^{{\nu }_2\left(n\right)}$, $\alpha_{1+i}\left(n\right)=2\nu_2\left(n\right)$, $\alpha_z\left(n\right)=\alpha_{\overline{z}}\left(n\right)=\nu_{{\left|z\right|}^2}\left(n\right)$ for all $z\in Z^+_{\mathbb{P}}$ and $\alpha_p\left(a\right)=\nu_p\left(a\right)$ for all $p\in {\mathcal{P}}_{3,4}$.
\end{remark}

\begin{corollary}{:}
\label{c1}
Let $X,Y$ be integers such that $X^2+Y^2=AB\in 2{\mathbb{N}}+1$ with $\mathcal{D}\left(AB\right)\cap {\mathcal{P}}_{3,4}=\emptyset $. Then there exists two Gaussian integers $\left(z_A,z_B\right)$ such that $X+iY=z_Az_B$, with $A={\left|z_A\right|}^2$ and $B={\left|z_B\right|}^2$. Furthermore, if $X,Y$ are coprime, this pair $\left(z_A,z_B\right)$ is unique up to a unit.
\end{corollary}

\paragraph{\textnormal{\underbar{\textit{Proof}}}}{:}
We have $\left(X+iY\right)\left(X-iY\right)=AB$. By uniqueness of the factorization into irreducible elements in $\mathbb{Z}\left[i\right]$ (theorem 1), we can write $X+iY=z_1z_2$ such that ${\left|z_1\right|}^2$ is a \underbar{maximal} divisor of $A$. By contradiction, suppose ${\left|z_1\right|}^2\neq A$. Then there exists $\pi\in Z_{\mathbb{P}}$ such that ${\left|z_1\right|}^2\pi$ divides $A$. Since $\pi\notin {\mathcal{P}}_{3,4} \cup \left\{1+i\right\}$, we have $\overline{\pi}\neq \pi\in Z_{\mathbb{P}}$ and $\overline{\pi}$ also divides $A$. Thus ${\left|z_1\pi\right|}^2$ divides $A$ which implies that ${\left|\pi\right|}^2$ divides $\frac{AB}{{\left|z_1\right|}^2}=\frac{X+iY}{z_1}.\frac{X-iY}{\overline{z_1}}$. It follows that $\pi$ or $\overline{\pi}$ must divide $\frac{X+iY}{z_1}$, which contradicts the maximality of ${\left|z_1\right|}^2$.
\\[10pt]
\noindent Thus ${\left|z_1\right|}^2=A$ and consequently ${\left|z_2\right|}^2=B$.
\\[10pt]
\noindent If, in addition, $X,Y$ are coprime, let $z'_A,z'_B$ be another pair satisfying the same properties. For any irreducible factor $\pi$ of $z_A$ we have ${\left|\pi\right|}^2 | A={\left|z'_A\right|}^2$ so either $\pi$ or $\overline{\pi}$ divides $z'_A$. However, if $\overline{\pi}$ divides $z'_A$, then $\pi|X+iY$ and $\overline{\pi}|X+iY$, which implies that ${\left|\pi\right|}^2|X+iY$. This would mean that $X$ and $Y$ share a common factor, which is a contradiction since they are coprime. Therefore $z_A$ and $z'_A$ share the same irreducible factors, from which it follows that there exists a unit $w$ such that $\left(z'_A,z'_B\right)=\left(wz_A,\overline{w}z_B\right)\blacksquare $

\noindent 
\subsection{Fermat's composite numbers}

\begin{proposition}{:}
\label{p1}
The set of prime divisors of $E^*_{Y^2}$ is given by:
\begin{center}
${\mathcal{D}}_p\left(E^*_{Y^2}\right)={\mathcal{P}}_{1,4}\backslash \mathcal{D}\left(Y\right).$
\end{center}
\end{proposition}

\paragraph{\textnormal{\underbar{\textit{Proof}}}}{:}
Let $p\in {\mathcal{D}}_p\left(E^*_{Y^2}\right)$. Then there exists an integer $X$ coprime to $Y$ such that $p|X^2+Y^2$. In particular, we necessarily have $p\equiv 1\ \left[4\right]$ and $p$ does not divide $Y$.

\noindent Conversely, if $p\in {\mathcal{P}}_{1,4}\backslash \mathcal{D}\left(Y\right)$ then Theorem 2 of \cite{a40} ensures the existence of an integer $X$ such that $p|X^2+Y^2$. This remains true if we replace $X$ with $X+\alpha p$. Since $p$ and $Y$ are coprime we can choose $\alpha$ such that $X+\alpha p$ is also coprime to Y (for instance, by choosing $\alpha$ such that $\alpha p\equiv -X+1\ \left[Y\right]$).$\blacksquare $

\begin{definition}{:}

\begin{enumerate}
\item  Let $F=\left\{X^2+Y^2,\ \left(X,Y\right)\in \mathbb{Z}\mathrm{\times }\mathbb{Z}\right\}$ be the set of sums of two squares.

\item  For $x,y\in \left\{0,1\right\}$ we define $F^{\left(x,y\right)}=\left\{X^2+Y^2,\ \left(X,Y\right)\in \left(2\mathbb{Z}\mathrm{+}x\right)\times \left(2\mathbb{Z}\mathrm{+}y\right)\right\}.$

\item  Let $G=\left\{n\in {\mathbb{N}}^*\mathrel{\left|\vphantom{n\in {\mathbb{N}}^* \forall p\in {\mathcal{D}}_p\left(n\right),\ p\in {\mathcal{P}}_{3,4}\ and\  \nu_p\left(n\right)\in 2\mathbb{N}}\right.\kern-\nulldelimiterspace}\forall p\in {\mathcal{D}}_p\left(n\right),\  p\in {\mathcal{P}}_{3,4} \mathrm{\ and\ } \nu_p\left(n\right)\in 2\mathbb{N}\right\}$ be the set of squares whose prime factors are congruent to 3 modulo 4.

\item  Let $H=\left\{2^\alpha,\ \alpha\in \mathbb{N}\right\}$.

\item  Finally, let $\Psi\boldsymbol{=}\bigcup_Y{E^*_{Y^2}}$.
\end{enumerate}
\end{definition}

\noindent We additionally observe that $F=F^{\left(0,1\right)}\cup F^{\left(0,0\right)}\cup F^{\left(1,1\right)},\ and\ F^{\left(0,1\right)}=F^{\left(1,0\right)}=F\cap \left(2\mathbb{Z}\mathrm{+}1\right)$. In this section, we establish further relations on $F$.

\begin{proposition}{:}
\label{p2}
$F^{\left(0,1\right)}=\Psi G$ and the decomposition of any element $s\in F^{\left(0,1\right)}$ into a product $s=\psi g$ with $\left(\psi,g\right)\in \Psi \times G$ is unique.
\end{proposition}

\paragraph{\textnormal{\underbar{\textit{Proof}}}}{:}
We know from \Cref{p1} that ${\mathcal{D}}_p\left(\Psi \right)={\mathcal{P}}_{1,4}$ and by construction ${\mathcal{D}}_p\left(G\right)={\mathcal{P}}_{3,4}$. The uniqueness of the prime factorization thus ensures that the decomposition $s=\psi g$, if it exists, is unique.
\\[10pt]
Conversely, let $s\in F^{\left(0,1\right)}$, and set $g=\prod_{p\in {\mathcal{P}}_{3,4}}{p^{\nu_p\left(s\right)}}$. Fermat's two-square theorem ensures that $g$ is a perfect square, hence $g\in G$.
\\[10pt]
\noindent Set $\psi=\frac{s}{g}$, we now show that it belongs to $\Psi$. By construction, all prime factors of $\psi$ belong to ${\mathcal{P}}_{1,4}$. The decomposition of $\psi$ into irreducible elements in $\mathbb{Z}\left[i\right]$ is given by:
\begin{center}
 $\psi=\prod_{p\in {\mathcal{P}}_{1,4}}{{\left(a_p+ib_p\right)}^{\nu_p\left(\psi\right)}{\left(a_p-ib_p\right)}^{\nu_p\left(\psi\right)}}$
\end{center}
where $a_p\ge b_p\in {\mathbb{N}}^*$ are the unique integers such that $p=a^2_p+b^2_p$. Let $a,b$ such that $a+ib=\prod_{p\in {\mathcal{P}}_{1,4}}{{\left(a_p+ib_p\right)}^{\nu_p\left(\psi\right)}}$. Then, by construction $\psi=a^2+b^2$. If $q$ were a common prime divisor common of $a$ and $b$ then $q$ would belong to ${\mathcal{P}}_{1,4}$. This would imply that $a_q-ib_q$ divides $a+ib$ in $\mathbb{Z}\left[i\right]$, which is impossible.
\\[10pt]
\noindent This proves that $F^{\left(0,1\right)}=\Psi G$.\textit{ }$\blacksquare $

\begin{proposition}{:}
\label{p3}
$F=F^{\left(0,1\right)}H=\Psi GH$.
Furthermore, every element $n\in F$ can be uniquely written as $n=\psi gh$ with $\left(\psi,g,h\right)\in \Psi \times G \times H$.
\end{proposition}

\paragraph{\textnormal{\underbar{\textit{Proof}}}}{:}
The result follows immediately by applying \Cref{p2} to the unique factorization $n=s\cdot 2^{{\nu }_2\left(n\right)}$, where $s\in F^{\left(0,1\right)}$ and $2^{{\nu }_2\left(n\right)}\in H$. $\blacksquare $

\noindent 
\subsection{Composite numbers of $\Psi$}

\noindent 
\subsubsection{Euler's Four-Square Identity }

\noindent In this section, we study the equation:
\[ \begin{array}{c}
A^2=B^2+C^2+D^2 \ \ \ \ \ \ \ \ \left(D_1\right) \end{array}
\] 
\begin{proposition}{:}
\label{p4}
We have the formal identity of $\mathbb{Z}\left[X,Y,Z,T\right]$:
\end{proposition}
\[ \begin{array}{c}
{\left(X^2+Y^2+Z^2+T^2\right)}^2= \\ 
{\left(X^2+Y^2-Z^2-T^2\right)}^2+4{\left(XZ-YT\right)}^2+4{\left(XT+YZ\right)}^2\ \ \ \ \ \ \ \left(D_2\right) \end{array}
\] 

\paragraph{\textnormal{\underbar{\textit{Proof}}}}{:}
Expanding the right-hand side, we obtain:
\begin{align*}
&{\left(X^2+Y^2-Z^2-T^2\right)}^2+4{\left(XZ-YT\right)}^2+4{\left(XT+YZ\right)}^2 \\
&=X^4+Y^4+Z^4+T^4+2\left(X^2Y^2-X^2Z^2-X^2T^2-Y^2Z^2-Y^2T^2+Z^2T^2\right) \\
&+4\left(X^2Z^2+Y^2T^2+X^2T^2+Y^2Z^2\right)+\left(8-8\right)XYZT \\
&={\left(X^2+Y^2+Z^2+T^2\right)}^2.
\end{align*}

\noindent This establishes the identity. $\blacksquare $
\\[10pt]
\noindent Thus, $\left(D_2\right)$ provides infinitely many solutions to $\left(D_1\right)$. Conversely, any \underbar{primitive} solution $\left(A,B,C,D\right)$ of $\left(D_1\right)$ \ $\left(that\ is,\ such\ that\ gcd\left(A,B,C,D\right)=1\right)$
can be parameterized in the form of $\left(D_2\right)$:

\begin{theorem}{\textbf{(Catalan \cite{a50}):}}
\label{t2}
For any \underbar{primitive solution} $\left(A,B,C,D\right)\in {\mathbb{N}}^4$ of $\left(D_1\right)$, there exists a quadruple ($x,y,z,t)\in {\mathbb{N}}^4$ such that, up to a permutation of $\left(B,C,D\right)$, we have:
\end{theorem}
\[ \begin{array}{c}
\left\{ \begin{array}{l}
A=x^2+y^2+z^2+t^2 \\ 
B=x^2+y^2-z^2-t^2 \\ 
C=2\left(xz-yt\right) \\ 
D=2\left(xt+yz\right) \end{array}
\right .\ \left(D_2\right) \end{array}
\] 
No permutation is needed if $B$ is odd.

\paragraph{\textnormal{\underbar{\textit{Proof}}}}{: }
We work again in the ring of Gaussian integers $\mathbb{Z}\left[i\right]$. Let $\left(A,B,C,D\right)$ be a solution of $\left(D_1\right)$ such that ${\mathrm{gcd} \left(A,B,C,D\right)\ }=1$. Therefore, at least one of them must be odd. Since any perfect square is congruent to 0 or 1 modulo 4, we deduce that $A$ must be odd, as well as exactly one of the integers among $B,C,D$. Up to a permutation, we can assume that $B$ is odd. The equation $\left(D_1\right)$ can then be rewritten as: 
\begin{center}
$\left(A+B\right)\left(A-B\right)=\left(C+iD\right)\left(C-iD\right).$
\end{center}
\noindent By uniqueness of the factorization into irreducible elements in $\mathbb{Z}\left[i\right]$, there exist $x,y,z,t\in \mathbb{Z}$ such that:
\begin{center}
$\left(C+iD\right)=2\left(x+iy\right)\left(z+it\right)$ \\
\end{center}
and:
\[ \begin{array}{c}
A+B=2\left(x+iy\right)\left(x-iy\right)=2\left(x^2+y^2\right) \\ 
A-B=2\left(z+it\right)\left(z-it\right)=2\left(z^2+t^2\right) \\ 
\end{array}
\] 
It immediately follow that:
\[ \begin{array}{c}
A=x^2+y^2+z^2+t^2 \\ 
B=x^2+y^2-z^2-t^2 \\ 
\end{array}
\] 
while:
\[ \begin{array}{c}
C=2\left(xz-yt\right), \\ 
D=2\left(xt+yz\right). \\ 
\end{array}
\] 

\noindent This completes the proof. $\blacksquare $

\begin{note}{:}
See also \cite{a60} for an alternative proof.
\end{note}

\begin{note}{:}
The primitivity hypothesis is essential, as shown by the counterexample $\left(9,-3,6,6\right)~$: if it could be written in the form $\left(D_2\right)$ it would follow that $3=x^2+y^2$, which is impossible. However, the hypothesis can be relaxed by merely assuming that $gcd\left(A,B,C,D\right)$ is a perfect square.
\end{note}

\begin{corollary}{:}
\label{c2}
For any solution of $\left(x,y,z,t\right)$ of $\left(D_2\right)$, we observe that $\left(-x,-y,-z,-t\right)$, $\left(-y,x,t,-z\right)$ and $\left(y,-x,-t,z\right)$ are also solutions. If we impose $x\ge 0$ and $\nu_2\left(x\right)\ge \nu_2\left(y\right)$, the quadruple is \underbar{unique} provided that $\frac{A+B}{2}$ is not a perfect square, ${\nu }_2\left(\frac{A+B}{2}\right)$ is even and ${\mathrm{gcd} \left(\frac{A+B}{2},\frac{A-B}{2}\right)\ }=1$.
\end{corollary}

\paragraph{\textnormal{\underbar{\textit{Proof}}}}{:}
Multiplying the factors $\left(x+iy\right)$ and $\left(z+it\right)$ by $-1$, or by $\pm i$ and $\mp i$ respectively, we obtain the three additional solutions. Imposing the conditions $x\ge 0$ and ${\nu }_2\left(x\right)\ge {\nu }_2\left(y\right)$ allows us to single out a unique solution from the four whenever $x,y\neq 0$ and ${\nu }_2\left(x\right)\neq {\nu }_2\left(y\right)$ -- and \textit{a fortiori} when  $\frac{A+B}{2}$ is not a perfect square and ${\nu }_2\left(\frac{A+B}{2}\right)$ is even. In addition, if  $\frac{A+B}{2}$ and $\frac{A-B}{2}$ are coprime, their irreducible factors in $\mathbb{Z}\left[i\right]$ form two disjoint sets. Thus, up to a unit, the decomposition $C+iD=2\left(x+iy\right)\left(z+it\right)$ such that $x+iy|\frac{A+B}{2}$ and $z+it|\frac{A-B}{2}$ is \underbar{unique}. $\blacksquare $

\noindent 
\subsubsection{The quadruple (p,q,r,s)}

\noindent In Proposition 2.5 of \cite{a1}, we showed that any factorization of an odd integer $n=AB$ corresponds to a representation of this number as a difference of two squares $n=U^2-V^2$. In particular, if $n=\left(1+4k\right)\left(1+4k'\right)$ is in $E_{Y^2}$, it can be written $n=X^2+Y^2=U^2-V^2$ with $U=2\left(k+k'\right)+1$ and $V=2\left(k-k'\right)$, which yields the relation $X^2+Y^2+V^2=U^2$. \underbar{Proposition 5} below provides another relation of this form.

\begin{proposition}{:}
\label{p5}
For any factorization $X^2+Y^2=\mathcal{Q}\left(h,N,Y\right)=AB$ of an element of $E_{Y^2}$, the triplet $\left(T,M,L\right)\in \left(2\mathbb{Z}\mathrm{+}1\right)\times \left(2\mathbb{Z}\mathrm{+}1\right)\times 2^{2h-1}\left(2\mathbb{Z}\mathrm{+}1\right)$ given by:

\begin{center}
$\left\{ \begin{array}{l}
T=A+2^{2h-1}B-2^hX \\ 
M=2^hX-A \\ 
L=2^{2h-1}B-2^hX \end{array}
\right.$
\end{center}

\noindent satisfies:

\begin{center}
$ \begin{array}{c}
\ \ \ \ \ \ \ \ \ T^2=M^2+L^2+{\left(2^hY\right)}^2.\ \left(Q\right) \end{array}
$
\end{center}
\end{proposition}

\paragraph{\textnormal{\underbar{\textit{Proof}}}}{:}
We have the following formal identity in $\mathbb{Z}\left[x,y,a,b\right]$:
\begin{center}
${\left(x-a\right)}^2+{\left(x-b\right)}^2+y^2={\left(x-a-b\right)}^2+x^2+y^2-2ab.$
\end{center}
\noindent Knowing that $X^2+Y^2=AB$, we let $x=2^hX$, $y=2^hY$, $a=A$ and $b=2^{2h-1}B$ so that the term $x^2+y^2-2ab$ vanishes.

\noindent The proposed triplet thus satisfies $\left(Q\right)$. Furthermore, since $A$ and $B$ are odd and $h\ge 1$, it follows that $T$ and $M$ are also odd, while $L=2^{2h-1}\left(B-2\left(2N+1\right)\right)\in 2^{2h-1}\left(2\mathbb{Z}\mathrm{+}1\right)$. $\blacksquare $
 
\begin{remark}{:}
\label{r3}
In \Cref{p5}, we require $h=\nu_2\left(X\right)$ but the identity $\left(Q\right)$ holds for any value of $h\in {\mathbb{N}}^*$.
\end{remark}

\begin{remark}{:}
\label{r4}
$\left(Q\right)$ Implies that $\left|T\right|>{\mathrm{max} \left(\left|M\right|,\left|L\right|,2^hY\right)}$. Since $T+M=2^{2h-1}B\ge 0$, we have $T\ge 0$.
\\[10pt]
\noindent Under the additional hypothesis ${\mathrm{gcd} \left(T,M,L,2^hY\right)\ }={\mathrm{gcd} \left(A,B,X,Y\right)\ }=1$, \Cref{t2} allows us to state that there exists a quadruple $\left(p,q,r,s\right)\in {\mathbb{Z}}^4$ such that: 
\begin{center}
$\left\{ \begin{array}{l}
T=p^2+q^2+r^2+s^2, \\ 
\left|M\right|=p^2+q^2-r^2-s^2, \\ 
\left|L\right|=2\left(pr-qs\right), \\ 
2^hY=2\left(ps+qr\right). \end{array}
\right.$
\end{center}
\end{remark}

\noindent By replacing $\left(p,q,r,s\right)$ with $\left(r,s,p,q\right)$ if necessary, we can assume $M=p^2+q^2-r^2-s^2$. Then, by replacing $\left(p,q,r,s\right)$ with $\left(-p,q,r,-s\right)$ if necessary, we can also assume $L=2\left(pr-qs\right)$.
\\[10pt]
\noindent Finally, per \Cref{c2} we can also assume that $p\ge 0$ and ${\nu }_2\left(p\right)\ge {\nu }_2\left(q\right)$, however the unicity conditions are not necessarily verified.
\\[10pt]
\noindent The following propositions establish several key properties on $\left(p,q,r,s\right)$.

\begin{proposition}{:}
\label{p6}
Retain assumptions of \Cref{p5}, and additionally assume ${\mathrm{gcd} \left(A,B,X,Y\right)\ }=1$. Consider a quadruple $\left(p,q,r,s\right)$ such that:

\begin{center}
$ \begin{array}{c}
\left\{ \begin{array}{l}
T=p^2+q^2+r^2+s^2\ , \\ 
M=p^2+q^2-r^2-s^2, \\ 
L=2\left(pr-qs\right), \\ 
2^hY=2\left(ps+qr\right) \\ 
p\ge 0 \\ 
\nu_2\left(p\right)\ge \nu_2\left(q\right) \end{array}
\right.\ \ \left(P_0\right) \end{array}
$
\end{center}
\noindent Then we have:$ \begin{array}{c}
p^2+q^2\neq 0\ \mathrm{and\ }\ r^2+s^2\neq 0.\ \ \left(P_1\right) \end{array}
$
\\[10pt]
\noindent If, in addition, $B\neq 1$, then we also have:
\begin{center}
$ \begin{array}{c}
p^2+q^2\boldsymbol{\neq }4^{h-1}\ \mathrm{and\ }\ r^2+s^2\boldsymbol{\neq }4^{h-1}.\ \ \left(P_2\right) \end{array}
$
\end{center}
\end{proposition}

\paragraph{\textnormal{\underbar{\textit{Proof}}}}{:}
From \Cref{r4}, we have $T>{\mathrm{max} \left(\left|M\right|,\left|L\right|,2^hY\right)\ }$ from which we deduce that $p^2+q^2=\frac{1}{2}\left(T+M\right)>0$ and $r^2+s^2=\frac{1}{2}\left(T-M\right)>0$.
\\[10pt]
\noindent The second property is equivalent to stating that $T+M$ and $T-M$ are both distinct from $2^{2h-1}$. By hypothesis, $T+M=2^{2h-1}B$ cannot be equal to $2^{2h-1}$ since $B\neq 1$. On the other hand:

\begin{enumerate}[label=$\bullet$]
\item  if $h=1$, then $2^{2h-1}=2~$ and $T-M=2\left(A+B-2X\right)$ is a multiple of 4 because $A+B-2X$ is even.

\item  if $h\ge 2$, then $2^{2h-1}$ is a multiple of 8 whereas $T-M=2A+2^{2h-1}B-2^{h+1}X\equiv 2A\ \left[8\right]$.
\end{enumerate}

\noindent In all cases, we conclude that $T-M\neq 2^{2h-1}$. $\blacksquare $

\begin{proposition}{:}
\label{p7}
Under the assumptions of \Cref{p6}, we have:
\begin{center}
$2N+1=2^{-h}X=4^{-h}\left(T+M-L\right)=4^{-h}\left(2A-T+M+L\right).$
\end{center}
\noindent In particular, $A=T-L={\left(p-r\right)}^2+{\left(q+s\right)}^2$
\\[10pt]
\noindent The proof is left to the reader.
\end{proposition}

\noindent 
\subsubsection{Classification of quadruples in the non-trivial case}

\noindent In this section, we keep working under the assumptions $\left(P_0\right)$ of \Cref{p6}.
\\[10pt]
\noindent We set $L'\coloneqq 2^{-\left(2h-1\right)}L$ i.e. $L'=B-2\left(2N+1\right)$. \Cref{p7} yields:
\[ \begin{array}{c}
2N+1=4^{-h}\left(T+M-L\right)=2^{-\left(2h-1\right)}\left(p^2+q^2\right)-\frac{L'}{2}.\ \ \ \ \ \ \left(P_3\right) \end{array}
\] 
Since $p^2+q^2=\frac{T+M}{2}=4^{h-1}B~$:

\begin{enumerate}[label=$\bullet$]
\item  when $h=1$, $p+q\equiv 1\ \left[2\right]$, and since ${\nu }_2\left(p\right)\ge {\nu }_2\left(q\right)$, $p$ is even and $q$ is odd;

\item  otherwise $4\ |\ p^2+q^2$, thus both $p$ and $q$ are even.
\end{enumerate}

\noindent Furthermore, given that $T=p^2+q^2+r^2+s^2$ is odd, the quadruple $\left(p,q,r,s\right)$ has one or three odd terms.
\\[10pt]
\noindent Finally, since $ps-qr=2^{h-1}Y~$:

\begin{enumerate}[label=$\bullet$]
\item  if $h=1$, $qr\equiv 1\ \left[2\right]$ which implies that $p$ is even and $q,s,r$ are odd,

\item  otherwise, only $s$ or $r$ is odd.
\end{enumerate}

\noindent 
\paragraph{\textit{\underline{1. Three odd terms}}}
\mbox{}\\
\mbox{}\\
\noindent We assume that $h=1$ in this section. Thus, $p$ is even and $q,r,s$ are odd.

\begin{proposition}{:}
\label{p8}
Set $u=q+s\in 2\mathbb{Z}$. We have:
\begin{center}
$2N+1=\frac{1}{2}\left(p\left(p\mathrm{-}r\right)\mathrm{+}qu\right)$
\end{center}
Either $1=\nu_2\left(p\right)<\nu_2\left(u\right)$ or $1=\nu_2\left(u\right)<\nu_2\left(p\right)$. 
\\[10pt]
Furthermore, if $X^2+Y^2\in E^*_{Y^2}$ and $A\neq 1$ (respectively $B\neq 1$) then $u\neq 0$ (respectively $p\neq 0$).
\end{proposition}

\paragraph{\textnormal{\underbar{\textit{Proof}}}}{:}
$L=2\left(pr-qs\right)$ implies that $L'=\frac{L}{2}=pr-qs$. $\left(P_3\right)$ yields:
\[2N+1=\frac{1}{2}\left(p^2+q^2-pr+qs\right)=\frac{1}{2}\left(p\left(p-r\right)+qu\right)\] 
Since $2N+1,p-r,q$ are odd and $p,u$ are even, necessarily $\frac{p}{2}$ and $\frac{u}{2}$ have opposite parity, which is equivalent to saying that $1=\nu_2\left(p\right)<\nu_2\left(u\right)$ or $1=\nu_2\left(u\right)<\nu_2\left(p\right)$.
\\[10pt]
\noindent Suppose $u=0$ i.e. $q=-s$. Then, it follows from $\left(P_0\right)$ that $A=T-L={\left(p-r\right)}^2$. If $X$ is coprime to $Y$, it is coprime to $X^2+Y^2=AB$, and thus coprime to $A$. Since $X=p\left(p-r\right)$ it follows that $p-r=\pm 1$, so $u=0$ implies $A=1$.
\\[10pt]
\noindent Finally, suppose $p=0$. $\left(P_0\right)$ yields $B=q^2$ and $Y=qr$. If $Y$ is coprime to $X$, as previously it is coprime to $B$ and $B$ must be equal to $1$. $\blacksquare $

\begin{proposition}{:}
\label{p9}
Set $a=\frac{pu}{4}\ and$ $b=\frac{qu}{2}$ and assume $u\neq 0$ (which is always true when $A\neq 1$). Then $b$ is non-zero and divides $a\left(4a-Y\right)$, $\frac{a\left(4a-Y\right)}{b}$ and $b$ have opposite parity and we can write:
\begin{center}
$\begin{array}{c}
2N+1=\frac{a\left(4a-Y\right)}{b}+b.\ \ \ \left(P_4\right) \end{array}$
\end{center}
If $A,B\neq 1$ we also have $a\neq 0$.
\end{proposition}

\paragraph{\textnormal{\underbar{\textit{Proof}}}}{:}
It is clear that $u\neq 0$ implies $b\neq 0$. We expand:
\[\frac{a\left(4a-Y\right)}{b}+b=\frac{p\left(pu-Y\right)}{2q}+\frac{qu}{2}=\frac{p\left(pq-qr\right)}{2q}+\frac{qu}{2}=\frac{1}{2}\left(p\left(p\mathrm{-}r\right)\mathrm{+}qu\right).\] 
\Cref{p8} then yields $\left(P_4\right)$. Since $2N+1$ is odd, $\frac{a\left(4a-Y\right)}{b}$ and $b$ must have opposite parities. Moreover, if $A,B\neq 1$ then we have $p,u\neq 0$ hence $a\neq 0$.\textit{ }$\blacksquare $

\begin{remark}{:}
There are two possible cases, either $\nu_2\left(a\right)=\nu_2\left(b\right)$, i.e. $\frac{p}{2}$ is odd and $\frac{u}{2}$ even, or $a,\frac{p}{2}$ even and $b,\frac{u}{2}$ odd. Since the product of $\frac{a\left(4a-Y\right)}{b}$ and $b$ is always even, we deduce that $a$ is always even. We can thus write: 
\begin{center}
$ \begin{array}{c}
2N+1=\frac{2\hat{a}\left(8\hat{a}-Y\right)}{\hat{b}}+\hat{b}\ \ \ \ \left(N_1\right) \end{array}
\ $ 
\end{center}
with $\hat{a}=\frac{1}{2}a$ and $\hat{b}=b$ both non-zero integers if $A,B\neq 1$.
\end{remark}

\noindent 
\paragraph{\textit{\underline{2. Three even terms}}}
\mbox{}\\
\mbox{}\\
\noindent We now assume $h>1$.

\begin{proposition}{:}
\label{p10}
If $h>1$, $r$ is odd and $\nu_2\left(q\right)=h-1<\nu_2\left(p\right)$. Furthermore, if $X^2+Y^2\in E^*_{Y^2}$ and $A\neq 1$ (resp $B\neq 1$) then $u\neq 0$ (resp $p\neq 0$).
\end{proposition}

\paragraph{\textnormal{\underbar{\textit{Proof}}}}{:}
Since $h>1$, $p$ and $q$ are even and $2N+1=\frac{1}{2}\left(\frac{p^2+q^2}{4^{h-1}}-L'\right)$, hence $\frac{p^2+q^2}{4^{h-1}}$ is odd. Let $\alpha$ be the largest integer such that $2^\alpha$ divides $p$ and $q$. $\frac{p^2+q^2}{4^{h-1}}$ being odd implies that $\alpha\le h-1$. But $\frac{p^2+q^2}{4^\alpha}$ is congruent to 1 or 2 modulo 4, so we must have $\alpha =h-1$. Hence $2^{h-1}$ divides $p$ and $q$, and $\nu_2\left(q\right)=h-1<\nu_2\left(p\right)$.
\\[10pt]
\noindent Let $q'\coloneqq \frac{q}{2^{h-1}}$, $q'$ is odd and $Y=q'r+\frac{p}{2^{h-1}}s\equiv r\ \left[2\right]$, thus $r$ is also odd.
\\[10pt]
\noindent If $u=0$ and $X^2+Y^2\in E^*_{Y^2}$, as in the proof of \Cref{p8}, we have $A=T-L={\left(p-r\right)}^2$ and $M=2^hX-A=p^2-r^2$, and $A,M$ are coprime, thus $p-r=\pm 1$  and $A=1$. By contrapositive, if $X^2+Y^2\in E^*_{Y^2}$ and $A\neq 1$, then $u\neq 0$.
\\[10pt]
\noindent Finally, if $p=0$ and $X^2+Y^2\in E^*_{Y^2}$, we have $Y=q'r$ and $4^{h-1}L'=qs$. $s=2^{h-1}s'$ yields $L'=q's'$. $\left(P_3\right)$ implies:
\begin{center}
$2N+1=2^{-\left(2h-1\right)}\left(p^2+q^2\right)-\frac{L'}{2}=\frac{1}{2}q'\left(q'-s'\right).$
\end{center}
\noindent $\mathrm{gcd}\left(2N+1,Y\right)=1$ implies $q'=\pm 1$ and $\left(P_2\right)$ from \Cref{p6} yields $B=1$. 
\\[10pt]
\noindent We conclude again by contrapositive that if $X^2+Y^2\in E^*_{Y^2}$ and $B\neq 1$, then $p\neq 0$. $\blacksquare $

\begin{notation}{:}
Set $u=q+s$ as in previous section. \Cref{p10} shows that when $h>1$, $u$ is even.
\end{notation}

\begin{proposition}{:}
\label{p11}
Set $a=\frac{pu}{2^{h+1}}$ and $b=\frac{qu}{2^h}$. They are integers and satisfy:
\begin{center}
$ \begin{array}{c}
2N+1=2^{-(h-1)}\left[\frac{a\left(4a-Y\right)}{b}+b\right] \ \ \ \  \left(P_4\right) \end{array}
$
\end{center}
Furthermore, $A,B\neq 1$ implies $a,b\neq 0$.
\end{proposition}

\paragraph{\textnormal{\underbar{\textit{Proof}}}}{:}
If $u\neq 0$, it is clear that b is non-zero. If $A,B\neq 1$, \Cref{p10} ensures that $a,b\neq 0$.
\\[10pt]
\noindent Furthermore, $\left(P_3\right)$ yields:
\begin{align*}
2N+1
&=2^{-\left(2h-1\right)}\left[\left(p^2+q^2\right)-pr+qs\right] \\
&=2^{-\left(2h-1\right)}\left[\frac{p^2\left(q+s\right)-p\left(ps+qr\right)}{q}+q\left(q+s\right)\right] \\ 
&=2^{-\left(2h-1\right)}\left[\frac{{\left(pu\right)}^2-pu.2^{h-1}Y}{qu}+qu\right] \\
&=2^{-\left(h-1\right)}\left[\frac{a\left(4a-Y\right)}{b}+b\right]
\end{align*}

\noindent This completes the proof. $\blacksquare $

\begin{note}{:}
\Cref{p11} generalizes \Cref{p9} and its equation $\left(P_4\right)$, which is why we have given it the same reference. In particular, it shows that b always divides $a\left(4a-Y\right)$.

\end{note}

\noindent \underbar{We distinguish three cases}:

\begin{enumerate}
\item  $\nu_2\left(p\right)=2\left(h-1\right)$

\item  $\nu_2\left(p\right)>2\left(h-1\right)$

\item  $h<\nu_2\left(p\right)<2\left(h-1\right)$
\end{enumerate}

\begin{corollary}{:}
\label{c3}
Suppose $\nu_2\left(p\right)=2\left(h-1\right)$. We have:
\[ \begin{array}{c}
2N+1=\frac{1}{2}\left(\frac{\tilde{a}\left(4^{h-1}\tilde{a}-Y\right)}{\tilde{b}}+\tilde{b}\right)\ \ \ \ \ \ \ \left(N_2\right) \end{array}
\] 
where $\tilde{a}=\frac{pu}{2^{3\left(h-1\right)}}$, $\tilde{b}\boldsymbol{=}\frac{qu}{2^{2\left(h-1\right)}}$ are both odd integers.
\end{corollary}

\paragraph{\textnormal{\underbar{\textit{Proof}}}}{:}
We have $a=2^{2h-4}\tilde{a}$ and $b=2^{h-2}\tilde{b}$. Thus 
\begin{center}
$2N+1=2^{-(h-1)}\left[\frac{a\left(4a-Y\right)}{b}+b\right]=2^{-(h-1)}\left[\frac{2^{2h-4}\tilde{a}\left(2^{2h-2}\tilde{a}-Y\right)}{2^{h-2}\tilde{b}}+2^{h-2}\tilde{b}\right]=\frac{1}{2}\left(\frac{\tilde{a}\left(4^{h-1}\tilde{a}-Y\right)}{\tilde{b}}+\tilde{b}\right)$.
\end{center}

\noindent Let us show that $\tilde{a},\tilde{b}$ are odd integers. We know that $\frac{a\left(4a-Y\right)}{b}$ is an integer and ${\nu }_2\left(\frac{a\left(4a-Y\right)}{b}\right)={\nu }_2\left(p\right)-{\nu }_2\left(q\right)-1=h-2$, thus $\frac{\tilde{a}\left(4^{h-1}\tilde{a}-Y\right)}{\tilde{b}}=\frac{a\left(4a-Y\right)}{2^{h-2}b}$ is an odd integer, and so is $\tilde{b}$ thanks to $\left(N_2\right)$. This implies that ${\nu }_2\left(u\right)=h-1$ and therefore that $\tilde{a}$ is also an odd integer. $\blacksquare $

\begin{corollary}{:}
\label{c4}
Suppose $\nu_2\left(p\right)>2\left(h-1\right)$. We have:
\[ \begin{array}{c}
2N+1=\frac{2\hat{a}\left(2^{2h+1}\hat{a}-Y\right)}{\hat{b}}+\hat{b}\ \ \ \ \ \ \ \ \left(N_3\right) \end{array}
\] 
with $\hat{a}=2^{-\left(2h-1\right)}a$ and $\hat{b}=2^{-(h-1)}b$ integers.
\end{corollary}

\paragraph{\textnormal{\underbar{\textit{Proof}}}}{:}
$\left(N_3\right)$ follows from $\left(P_4\right)$. Furthermore:
\begin{center}
$\nu_2\left(\frac{2\hat{a}\left(2^{2h+1}\hat{a}-Y\right)}{\hat{b}}\right)=\nu_2\left(\frac{a}{b}\right)-\left(h-1\right)=\nu_2\left(p\right)-\nu_2\left(q\right)-h\ge 0.$
\end{center}
We deduce that $\frac{2\hat{a}\left(2^{2h+1}\hat{a}-Y\right)}{\hat{b}}\in \mathbb{Z}$ from which it follows that $\hat{b}=2N+1-\frac{2\hat{a}\left(2^{2h+1}\hat{a}-Y\right)}{\hat{b}}\in \mathbb{Z}$. Thus $\nu_2\left(\hat{b}\right)\ge 0$ hence $\nu_2\left(u\right)\ge h$. Since $2N+1$ is odd we know that $\frac{2\hat{a}\left(2^{2h+1}\hat{a}-Y\right)}{\hat{b}}$ or $\hat{b}$ is even which implies that $2\hat{a}$ is even and $\hat{a}\in \mathbb{Z}$. $\blacksquare $

\begin{note}
$\left(N_3\right)$ is a generalization of $\left(N_1\right)$.
\end{note}

\begin{corollary}{:}
\label{c5}
Suppose $\nu_2\left(p\right)<2\left(h-1\right)$. We have:
\[ \begin{array}{c}
2N+1=2^{\nu_2\left(p\right)-\left(2h-1\right)}\left(\frac{\overline{a}\left(4^{\nu_2\left(p\right)-\left(h-1\right)}\overline{a}-Y\right)}{\overline{b}}+\overline{b}\right)\ \ \ \ \left(N_4\right) \end{array}
\] 
and $\overline{a}=4^{h-\nu_2\left(p\right)}a$, $\overline{b}\boldsymbol{=}2^{h-\nu_2\left(p\right)}b$ odd integers.
\end{corollary}

\paragraph{\textnormal{\underbar{\textit{Proof}}}}{:}
$\left(N_4\right)$ still follows from $\left(P_4\right)$. Furthermore,
\begin{center}
$\nu_2\left(\frac{\overline{a}}{\overline{b}}\right)=h-\nu_2\left(p\right)+\nu_2\left(p\right)-\nu_2\left(q\right)+1=0$
\end{center}
Thus, $\frac{\overline{a}\left(4^{\nu_2\left(p\right)-\left(h-1\right)}\overline{a}-Y\right)}{\overline{b}}$ is odd and so is $\overline{b}$. It follows that their product is odd, and finally that $\overline{a}$ is odd. $\blacksquare $

\noindent
\paragraph{\textit{\underline{3. Classification of non-trivial integers}}}
\mbox{}\\
\mbox{}\\
\noindent In the previous two sections, four equalities $\left(N_1\right) - \left(N_4\right)$ were derived by classifying $p$ into four cases. Here, we simplify them into two cases yielding $\left(V_1\right)$ and $\left(V_2\right)$ respectively.

\begin{proposition}{:}
\label{p12}
For any factorization $X^2+Y^2=\mathcal{Q}\left(h,N,Y\right)=AB$ of an element in $E^*_{Y^2}$ such that $A,B\neq 1$, and for any quadruple $\left(p,q,r,s\right)$ satisfying $\left(P_0\right)$, one of the following two cases holds:

\begin{enumerate}[label=\alph*)]
\item  If $\nu_2\left(p\right)>2\left(h-1\right)$, either $\left(N_1\right)$ or $\left(N_3\right)$ holds, i.e. there exist $\hat{a},\hat{b}\in {\mathbb{Z}}^*$ such that ${2\hat{a}\left(2\cdot 4^h\hat{a}-Y\right)}/{\hat{b}}$ and $\hat{b}$ have opposite parities and:
\[ \begin{array}{c}
2N+1=\left(\frac{2\hat{a}\left(2\cdot 4^h\hat{a}-Y\right)}{\hat{b}}\right)+\hat{b}\ \ \ \ \ \ \ \ \left(V_1\right) \end{array}
\] 

\item  Otherwise either $(N_2)$ or $\left(N_4\right)$ holds and there exists $\check{a},\check{b}\in 2\mathbb{Z}\mathrm{+}1$ and $1\le \check{h}\le h-1$ such that:
\[ \begin{array}{c}
2^{h-\check{h}}\left(2N+1\right)=\frac{\check{a}\left(4^{\check{h}}\check{a}-Y\right)}{\check{b}}+\check{b}\ \ \ \ \ \left(V_2\right) \end{array}
\] 
\end{enumerate}
\end{proposition}

\noindent The proof follows largely from the preceding results and is left to the reader.

\begin{note}{:}
$\left(V_1\right)$ can also be viewed as a special case of $\left(V_2\right)$ by setting $\check{h}=h$ and $\check{a}=2\hat{a}$, and relaxing the condition that the variables must be odd.
\\[10pt]
\noindent Since $\check{b}$ divides $\check{a}\left(4^{\check{h}}\check{a}-Y\right)$, it can be factored as $\check{b}=b_1b_2$ with $b_1|\check{a}$ and $b_2|4^{\check{h}}\check{a}-Y$. This allows us to write $2N+1=a_1a_2+b_1b_2$ where $a_1=\frac{\check{a}}{b_1}$ and $a_2=\frac{4^{\check{h}}\check{a}-Y}{b_2}$.
\end{note}

\begin{note}{:}
If $h=1$, we always fall into the case $\left(V_1\right)$ as the condition $\nu_2\left(p\right)>2\left(h-1\right)$ is automatically satisfied. $\blacksquare $
\end{note}

\noindent 
\subsubsection{Extension to certain trivial integers}

\noindent In this section, we establish the following lemma, which extends \Cref{p12} to the case of trivial composite integers having a sufficient number of factors in ${\mathcal{P}}_{1,4}$:

\begin{lemma}{:}
\label{l1}
Trivial \underbar{composite integers} $n=X^2+Y^2\in E_{Y^2}\diagdown \left\{Y^2\right\}$ such that:
\begin{center}
$\sum_{p\in {\mathcal{P}}_{1,4}}{\nu_p\left(n\right)}\ge 2$
\end{center}
satisfy equation $\left(V_1\right)$ or $\left(V_2\right)$. That is, if we write $X=2^h\left(2N+1\right)$, either there exist $\hat{a},\hat{b}\in {\mathbb{Z}}^*$ such that:
\begin{center}
$ \begin{array}{c}
2N+1=\left(\frac{2\hat{a}\left(2\cdot 4^h\hat{a}-Y\right)}{\hat{b}}\right)+\hat{b}, \end{array} $
\end{center}
or there exists $\check{a},\check{b}\in 2\mathbb{Z}\mathrm{+}1$ and $1\le \check{h}\le h-1$ such that:
\[ \begin{array}{c}
2^{h-\check{h}}\left(2N+1\right)=\frac{\check{a}\left(4^{\check{h}}\check{a}-Y\right)}{\check{b}}+\check{b}. \end{array}
\] 
\end{lemma}

\paragraph{\textnormal{\underbar{\textit{Proof}}}}{:}
We know that to any factorization $n=X^2+Y^2=AB$ such that ${\mathrm{gcd} \left(A,B,X,Y\right)\ }=1$ we can associate a quadruple $\left(p,q,r,s\right)$ from which we can construct an equality $\left(V_1\right)$ or $\left(V_2\right)$ \underbar{provided that} $pu\neq 0$.
\\[10pt]
\noindent Let $n=X^2+Y^2$ be a trivial composite number such that $\sum_{p\in {\mathcal{P}}_{1,4}}{\nu_p\left(n\right)}\ge 2$.
\\[10pt]
\noindent We set $d={\mathrm{gcd} \left(X,Y\right)\ }>1$, $X=dx$, $Y=dy$. If $x^2+y^2$ is composite, we can apply \Cref{p12} and deduce that one of the following two assertions holds:

\begin{enumerate}
\item  There exists $\left(\hat{a},\hat{b}\right)\in {\left({\mathbb{Z}}^*\right)}^2$ such that  $\frac{2N+1}{d}=\left(\frac{2\hat{a}\left(2\cdot 4^h\hat{a}-y\right)}{\hat{b}}\right)+\hat{b}$ from which we obtain $2N+1=\left(\frac{2\hat{a}d\left(2\cdot 4^h\hat{a}d-Y\right)}{\hat{b}d}\right)+\hat{b}d$.

\item  There exists $\left(\hat{a},\hat{b}\right)\in {\left({2\mathbb{Z}+1}\right)}^2$ and $1\le \check{h}\le h-1$ such that $2^{h-\check{h}}\frac{2N+1}{d}=\frac{\check{a}\left(4^{\check{h}}\check{a}-y\right)}{\check{b}}+\check{b}$ from which we obtain $2^{h-\check{h}}\left(2N+1\right)=\frac{\check{a}d\left(4^{\check{h}}\check{a}d-Y\right)}{\check{b}d}+\check{b}d$.
\end{enumerate}

\noindent In both cases, the lemma is verified. For the remainder of the proof, we can therefore assume that $x^2+y^2=\pi$ is a prime number.
\\[10pt]
\noindent We recall the equations linking $\left(X,Y,A,B\right)$ with $\left(p,q,r,s\right)$:
\[\left\{ \begin{array}{l}
2^{2\left(h-1\right)}B=p^2+q^2\ , \\ 
A={\left(p-r\right)}^2+{\left(q+s\right)}^2, \\ 
2^{h-1}X=p^2+q^2-pr+qs, \\ 
2^{h-1}Y=ps+qr \\ 
p\ge 0 \\ 
\nu_2\left(p\right)\ge \nu_2\left(q\right) \end{array}
\right.\] 
The condition $p=0$ implies that $B$ is the square of a number dividing $X$, while $u=q+s=0$ implies that $A$ is the square of a number dividing $Y$.
\\[10pt]
\noindent Consequently, we are left with the case where $X^2+Y^2=\pi d^2$ with $\pi={\left|x+iy\right|}^2$ and $d={\mathrm{gcd} \left(X,Y\right)\ }$ is a multiple of a prime number $\pi'\in {\mathcal{P}}_{1,4}$. Let $\alpha,\ \beta \in \mathbb{N}^*$ such that  $\pi'={\left|\alpha+i\beta\right|}^2$. We write $d=\pi'd'$ so that $d^2={\left|\left(\alpha^2-\beta^2\right)d'+2i\alpha \beta d'\right|}^2$.
\\[10pt]
\noindent Set $A=\pi$ and $B=d^2$. Let us show that the quadruple:

\begin{center}
$\left(p=2^hd'\alpha\beta,\ q=2^{h-1}d'\left(\alpha^2-\beta^2\right),r=2^hd'\alpha\beta-y,s=x-2^{h-1}d'\left(\alpha^2-\beta^2\right)\right)$
\end{center}

\noindent is a solution to the system above, with $p\neq 0$:

\begin{enumerate}[label=$\bullet$]
\item  $p^2+q^2=2^{2\left(h-1\right)}d^{'2}{\left(\alpha^2+\beta^2\right)}^2=2^{2\left(h-1\right)}B$

\item  ${\left(p-r\right)}^2+{\left(q+s\right)}^2=y^2+x^2=\pi=A$

\item  $p^2+q^2-pr+qs=2^{h-1}d'\left(\alpha^2+\beta^2\right)x=2^{h-1}X$

\item  $ps+qr=2^{h-1}d'\left(\alpha^2+\beta^2\right)y=2^{h-1}Y$

\item  $p>0$ by definition

\item  $\nu_2\left(p\right)=\nu_2\left(q\right)+1+\nu_2\left(\alpha\beta\right)\ge \nu_2\left(q\right)$ because $\alpha^2-\beta^2$ must be odd.
\end{enumerate}

\noindent Moreover $u\neq 0$ because $A=\pi$ is not a square. That's what we needed to prove. $\blacksquare $
\\[10pt]
\noindent \underbar{Remark}~: the quadruple $\left(p=0,q=2^{h-1}d,r=y,s=x-q\right)$ is also a suitable choice for the last case studied, illustrating the non-uniqueness of the solutions to $\left(P_0\right)$.

\begin{remark}{:}
Equations $\left(V_1\right)$ or $\left(V_2\right)$ lead to alternative decompositions of $n$ as sum of two squares. More precisely, if there exist $\hat{a},\hat{b}$ satisfying $\left(V_1\right)$ then:
\begin{center}
$n=X^2+Y^2={\left(X-2^{h+1}\hat{b}\right)}^2+{\left(Y-4^{h+1}\hat{a}\right)}^2$
\end{center}
and if there exist $\check{a},\check{b},\check{h}$ satisfying $\left(V_2\right)$ then:
\begin{center}
$n=X^2+Y^2={\left(X-2^{\check{h}+1}\check{b}\right)}^2+{\left(Y-2.4^{\check{h}}\check{a}\right)}^2.$\textit{}
\end{center}
\end{remark}

\paragraph{\textnormal{\underbar{\textit{Proof}}}}{:}
Expanding the expressions yields the result. In the case $\left(V_1\right)$ we have $2N+1=\left(\frac{2\hat{a}\left(2\cdot 4^h\hat{a}-Y\right)}{\hat{b}}\right)+\hat{b}$, $X=2^h\left(2N+1\right)$ and:
\begin{center}
${\left(X-2^{h+1}\hat{b}\right)}^2+{\left(Y-4^{h+1}\hat{a}\right)}^2=X^2+Y^2+4^{h+1}\hat{b}\left(\hat{b}-\left(2N+1\right)+\frac{4^{h+1}{\hat{a}}^2-2\hat{a}Y}{\hat{b}}\right)=n.$
\end{center}

\noindent The proof is analogous for case $\left(V_2\right)$. $\blacksquare $
\\[10pt]
\noindent Together with the results that will be established in Part 3, this implies that if $\sum_{p\in {\mathcal{P}}_{1,4}}{\nu_p\left(n\right)}\le 1$ then equations $\left(V_1\right)$ and $\left(V_2\right)$ are never satisfied, thereby establishing the converse of \Cref{l1}.

\noindent 
\section{Quadratic Sequences on X and the identity of Diophantus}

\noindent In this section, we use the results from section 2 to construct quadratic sequences $\left(X_i\right)$ such that $n_i\coloneqq X^2_i+Y^2$ is composite in $E^*_{Y^2}$.

\begin{definition}{:}
Let $E^1_{Y^2}$ (respectively $E^2_{Y^2}$) be the set of composite numbers of $E^*_{Y^2}$ which fall into the case $\left(V_1\right)$ (resp $\left(V_2\right)$). In other words, $n=X^2+Y^2$ is in $E^1_{Y^2}$ (resp. $E^2_{Y^2}$) if and only if it is a non-trivial composite $n=AB$, with $A,B>1$, such that there exists a quadruple $\left(p,q,r,s\right)$ verifying equations $\left(P_0\right)$ and $\nu_2\left(p\right)>2\left(h-1\right)$ (resp. $\nu_2\left(p\right)\le 2\left(h-1\right)$).\textbf{}
\\[10pt]
\noindent According to the results of Part 2, $E^1_{Y^2}\cup E^2_{Y^2}=E^*_{Y^2}\diagdown \mathbb{P}$.
\end{definition}

\noindent 
\subsection{Sequences on X}

\noindent Throughout this section, let $n=X^2+Y^2\in E^*_{Y^2}$ be a composite number, with a factorization $n=AB$ where $A,B>1$. Notations from Part 2 are also retained.

\noindent 
\subsubsection{Factors A and B as sums of two squares}

\noindent Recall the identity $2N+1=2^{-(h-1)}\left[\frac{a\left(4a-Y\right)}{b}+b\right]$. Since $b$ divides $a\left(4a-Y\right)$ we can write this identity as:
\begin{center}
$X=2\left(a_1a_2+b_1b_2\right)$
\end{center}
\noindent with a factorization $b=b_1b_2$ such that $a_1=\frac{a}{b_1}$ and $a_2=\frac{4a-Y}{b_2}$ are integers. In particular, we have $Y=4a-a_2b_2=4a_1b_1-a_2b_2$. The four integers $a_1,a_2,b_1,b_2$ must be non-zero.
\\
\noindent From this, we deduce the \underbar{non-trivial factorization}:
\begin{align*}
X^2+Y^2
&=4\left(a^2_1a^2_2+b^2_1b^2_2+2a_1a_2b_1b_2\right)+\left(16a^2_1b^2_1+a^2_2b^2_2-8a_1a_2b_1b_2\right) \\
&=4a^2_1a^2_2+4b^2_1b^2_2+16a^2_1b^2_1+a^2_2b^2_2 \\
&=\left(4a^2_1+b^2_2\right)\left(a^2_2+4b^2_1\right)
\end{align*}

\noindent Since $b_1,b_2$ are not uniquely determined, we cannot necessarily identify this factorization with the initial factorization $X^2+Y^2=AB$. Moreover, from $X^2+Y^2=AB$ and ${\mathrm{gcd} \left(X,Y\right)\ }=1$ and Fermat’s two squares theorem, it follows that $A$ and $B$ can be written as the sum of two non-zero squares, one even and the other odd. We then find another identity $X^2+Y^2=\left(4a^2_1+b^2_2\right)\left(a^2_2+4b^2_1\right)$ where now $A=4a^2_1+b^2_2$, $B=a^2_2+4b^2_1$ through we no longer necessarily have $X=2\left(a_1a_2+b_1b_2\right)$ or $Y=4a_1b_1-a_2b_2$ (and again the choice of $a_1,a_2,b_1,b_2$ is not unique). 
\\
The following proposition, however, shows that we can satisfy these four conditions simultaneously, for a unique choice of integers:

\begin{proposition}{:}
\label{p13}
There exists a quadruple $\left(a_1,a_2,b_1,b_2\right)\in {{\mathbb{Z}}^*}^4$, unique up to the sign, such that the following equations are verified:
\[ \begin{array}{c}
 \begin{array}{cc}
X=2\left(a_1a_2+b_1b_2\right), & Y=4a_1b_1-a_2b_2, \\ 
A=4a^2_1+b^2_2, & B=a^2_2+4b^2_1. \end{array}
\ \ \ \ \left(P'_0\right) \end{array}
\] 
\end{proposition}

\paragraph{\textnormal{\underbar{\textit{Proof}}}}{:}
With Gaussian integers, we factor $X+iY$ into a product of irreducible factors $\prod_k{{\left(p_k+iq_k\right)}^{\nu_k}}$, which yields $X^2+Y^2={\left|X+iY\right|}^2=\prod_k{{\left|p_k+iq_k\right|}^{2\nu_k}}$.
\\[10pt]
\noindent It follows that there exist integers $\alpha_k,\beta_k$ such that $A=\prod_k{{\left|p_k+iq_k\right|}^{2\alpha_k}}$, $B=\prod_k{{\left|p_k+iq_k\right|}^{2\beta_k}}$ and $\alpha_k+\beta_k=\nu_k$.
\\[10pt]
\noindent In particular, by setting (up to multiplication by a unit) $2a_1-ib_2=\prod_k{{\left(p_k+iq_k\right)}^{\alpha_k}}$ and $a_2+2ib_1=\prod_k{{\left(p_k+iq_k\right)}^{\beta_k}}$, we obtain $X+iY=\left(2a_1-ib_2\right)\left(a_2+2ib_1\right)$. This implies $X=2\left(a_1a_2+b_1b_2\right)$ and $Y=4a_1b_1-a_2b_2$, while by construction we also have $A=4a^2_1+b^2_2$ and $B=a^2_2+4b^2_1$.
\\[10pt]
\noindent Moreover, let us show that all four $\left(a_1,a_2,b_1,b_2\right)$ are non-zero. First, $a_2$ and $b_2$ must be odd because $A$ and $B$ are. Second, if $a_1=0$ then $\left|b_2\right|>1$ (because $A\neq 1$), which would make $b_2$ a common divisor of $X=2b_1b_2$ and $Y=-a_2b_2$ contradicting ${\mathrm{gcd} \left(X,Y\right)\ }=1$. Last, if $b_1=0$ then similarly $\left|a_2\right|>1$ would be a common divisor to $X,Y$, which is impossible.
\\[10pt]
\noindent Finally, this representation is unique up to sign. Indeed, if $X+iY=\left(2a'_1-ib'_2\right)\left(a'_2+2ib'_1\right)$ with $A={\left|2a'_1-ib'_2\right|}^2$ and $B=|a'_2+2ib'_1|\textrm{²}$ then the irreducible factors of $2a_1-ib_2$ and $2a'_1-ib'_2$ must be the same up to conjugation. However $X+iY$ cannot have conjugate irreducible factors (otherwise $X$ and $Y$ would not be coprime). This implies that these complex numbers are equal up to sign, which completes the proof. $\blacksquare $

\begin{example}{:}
\begin{enumerate}[label=\alph*.]
\item  With $n=5185={44}^2+{57}^2=17\times \left(5\times 61\right)$ we have $44+57i=\left(2+i\right)\left(4-i\right)\left(6+5i\right)$. We derive $2a_1-ib_2=4-i$ and $a_2+2ib_1=\left(2+i\right)\left(6+5i\right)=7+16i$, i.e. $\left(a_1,a_2,b_1,b_2\right)=\left(2,7,8,1\right)$.

\item  With $n=85=2^2+9^2=5\times 17$ we have $2+9i=i\left(2-i\right)\left(4+i\right)$. We derive $2a_1-ib_2=2-i$ and $a_2+2ib_1=-1+4i$, i.e. $\left(a_1,a_2,b_1,b_2\right)=\left(1,-1,2,1\right)$.
\end{enumerate}
\end{example}

\begin{corollary}{:}
\label{c6}
Let be $\mathcal{F}=\left\{\left(X,Y,A,B\right)\in 2\mathbb{N}\mathrm{\times }\left(2\mathbb{N}\mathrm{+}1\right)\mathrm{\times }\mathbb{N}\mathrm{\times }\mathbb{N}\mathrm{|}{\mathrm{pgcd} \left(X,Y\right)\ }=1,\ A,B>1\right\}$ the set of decompositions of non-trivial Fermat numbers. Then there exists a unique function
\begin{center}
\noindent $\mathcal{D}\ :\ \  \begin{array}{ccc}
\mathcal{F} & \to  & {\mathbb{Z}}^4 \\ 
\left(X,Y,A,B\right) & \mapsto  & \left(a_1,a_2,b_1,b_2\right) \end{array}
$
\end{center}
such that $\left(a_1,a_2,b_1,b_2\right)$  is the unique quadruple satisfying $\left(P'_0\right)$ and $b_1>0$.
\end{corollary}

\begin{remark}{:}
By extension, we will denote the four components of $\mathcal{D}\left(X,Y,A,B\right)$ by $a_1\left(X,Y,A,B\right)$, $a_2\left(X,Y,A,B\right)$, $b_1\left(X,Y,A,B\right)$, $b_2\left(X,Y,A,B\right)$.
\end{remark}

\noindent 
\subsubsection{Sequences on the terms of Diophantus' identity}

\noindent We can obtain other non-trivial composite numbers in $E^*_{Y^2}$ by modifying the quadruple provided that $4a_1b_1-a_2b_2$ remains equal to $Y$.

\begin{proposition}{:}
\label{p14}
Let $y$ be the product of the prime numbers dividing $Y$ and $m={\mathrm{lcm} \left(a_2,4b_1,y\right)}$. The sequence
\begin{center}
$X_k=2\left(\left(a_1+\frac{m}{4b_1}k\right)a_2+b_1\left(b_2+\frac{m}{a_2}k\right)\right)$
\end{center}
is such that $X^2_k+Y^2$ is composite and belongs to $E^*_{Y^2}$, provided that $a_1+\frac{m}{4b_1}k\neq 0$. More precisely, for any such $k\in \mathbb{Z}$ (and especially for $\left|k\right|$ sufficiently large), setting $A_k=4{\left(a_1+\frac{m}{4b_1}k\right)}^2+{\left(b_2+\frac{m}{a_2}k\right)}^2$ yields:
\[X^2_k+Y^2=A_kB,\] 
\[\mathcal{D}\left(X_k,Y,A_k,B\right)=\left(a_1+\frac{m}{4b_1}k,a_2,b_1,b_2+\frac{m}{a_2}k\right).\] 
\end{proposition}

\paragraph{\textnormal{\underbar{\textit{Proof}}}}{:}
By construction $Y=4\left(a_1+\frac{m}{4b_1}k\right)b_1-a_2\left(b_2+\frac{m}{a_2}k\right)$. Thus:
\[X^2_k+Y^2=\left(4{\left(a_1+\frac{m}{4b_1}k\right)}^2+{\left(b_2+\frac{m}{a_2}k\right)}^2\right)\left(a^2_2+4b^2_1\right)\] 
Since $m$ is even, $X^2_k+Y^2$ is odd. Furthermore, since $a_2$ and $b_2$ are odd, $b_2+\frac{m}{a_2}k$ is odd and therefore non-zero. If, in addition, $a_1+\frac{m}{4b_1}k$ is non-zero, then $\left(4{\left(a_1+\frac{m}{4b_1}k\right)}^2+{\left(b_2+\frac{m}{a_2}k\right)}^2\right)>1$ and $X^2_k+Y^2$ is composite.
\\[10pt]
\noindent Finally, for any prime $p$ dividing $Y$, $p$ divides $m$ but does not divide $X$, and hence does not divide $X_k$. Thus $X_k$ is coprime to $Y$, i.e.  $X^2_k+Y^2\in E^*_{Y^2}$. $\blacksquare $

\begin{remark}{:}
Other sequences can be constructed similarly by adding arithmetic progressions to $\left(a_2,b_1\right),(a_1,a_2)$ or $\left(b_1,b_2\right)$.
\end{remark}

\noindent Call those defined in \Cref{p13} \textit{Type 1 sequences}, we can also define:

\begin{enumerate}[label=$\bullet$]
\item  \textit{Type 2 sequences}: let
\begin{align*}
&m={\mathrm{lcm} \left(4a_1,b_2,y\right)\ },\ k\in \mathbb{Z} \\
&X_k=2\left(a_1\left(a_2+\frac{m}{b_2}k\right)+\left(b_1+\frac{m}{4a_1}k\right)b_2\right),\\
&B_k={\left(a_2+\frac{m}{b_2}k\right)}^2+4{\left(b_1+\frac{m}{4a_1}k\right)}^2,
\end{align*}
We have $\mathcal{D}\left(X_k,Y,A,B_k\right)=\left(a_1,a_2+\frac{m}{b_2}k,b_1+\frac{m}{4a_1}k,b_2\right)$.

\item  \textit{Type 3 sequences}: let 
\begin{align*}
&m={\mathrm{lcm} \left(4b_1,b_2,y\right)\ },\ k\in \mathbb{Z} \\
&X_k=2\left(\left(a_1+\frac{m}{4b_1}k\right)\left(a_2+\frac{m}{b_2}k\right)+b_1b_2\right),\\
&A_k=4{\left(a_1+\frac{m}{4b_1}k\right)}^2+b^2_2, \\
&B_k={\left(a_2+\frac{m}{b_2}k\right)}^2+4b^2_1,
\end{align*}
We have $\mathcal{D}\left(X_k,Y,A_k,B_k\right)=\left(a_1+\frac{m}{4b_1}k,a_2+\frac{m}{b_2}k,b_1,b_2\right)$.

\item  \textit{Type 4 sequences}: let 
\begin{align*}
&m={\mathrm{lcm} \left(4a_1,a_2,y\right)\ },\ k\in \mathbb{Z} \\
&X_k=2\left(a_1a_2+\left(b_1+\frac{m}{4a_1}k\right)\left(b_2+\frac{m}{a_2}k\right)\right),\\
&A_k=4a^2_1+{\left(b_2+\frac{m}{a_2}k\right)}^2, \\
&B_k=a^2_2+4{\left(b_1+\frac{m}{a_1}k\right)}^2,
\end{align*}
We have $\mathcal{D}\left(X_k,Y,A_k,B_k\right)=\left(a_1,a_2,b_1+\frac{m}{4a_1}k,b_2+\frac{m}{a_2}k\right)$.
\end{enumerate}

\noindent Note, however, that for the last two types, the corresponding sequence will be a quadratic sequence, meaning it describes a subset of $E^*_{Y^2}$ with an asymptotic density of zero, unlike the first two types.

\noindent 
\subsection{Relationship between quadruples $\left({\boldsymbol{a}}_{\boldsymbol{1}},{\boldsymbol{a}}_{\boldsymbol{2}},{\boldsymbol{b}}_{\boldsymbol{1}},{\boldsymbol{b}}_{\boldsymbol{2}}\right)$ and $\left(\boldsymbol{p},\boldsymbol{q},\boldsymbol{r},\boldsymbol{s}\right)$}

\noindent Recall that we have constructed the second quadruple, which is unique up to sign, as follows:
\[X^2+Y^2=\underbrace{{\left|2a_1-ib_2\right|}^2}_{=A}\underbrace{{\left|a_2+2ib_1\right|}^2}_{=B}.\] 
We further defined :
\begin{center}
$\left\{ \begin{array}{l}
T=A+2^{2h-1}B-2^hX \\ 
M=2^hX-A \\ 
L=2^{2h-1}B-2^hX \end{array}
\right.$
\end{center}
such that :
\begin{center}
$T^2=M^2+L^2+{\left(2^hY\right)}^2$
\end{center}
and we constructed a quadruple $\left(p,q,r,s\right)$ satisfying $\left(P_0\right)$. The constructions of both quadruples are based on factorizations in $\mathbb{Z}\left[i\right]$. Specifically, the following elements have been introduced:
\begin{center}
$z_A=2a_1-ib_2,z_B=a_2+2ib_1\in \mathbb{Z}\left[i\right],\ \ \ X+iY=z_Az_B,\ A=z_A\overline{z_A},\ B=z_B\overline{z_B},$
\end{center}
\noindent and:
\begin{center}
$z_P=p+iq,\ z_R=r+is\in \mathbb{Z}\left[i\right],\ \ L+2^hiY=2z_Pz_R,\frac{T+M}{2}=z_P\overline{z_P},\frac{T-M}{2}=z_R\overline{z_R}$
\end{center}

\noindent We notice that:
\begin{align*}
L+2^hiY &=2^{2h-1}B-2^h\left(X-iY\right) \\
&=2^{2h-1}z_B\overline{z_B}-2^h\overline{z_Az_B} \\
&=\left(2^{2h-1}z_B-2^h\overline{z_A}\right)\overline{z_B}
\end{align*}

\noindent However, the factorization $L+2^hiY=2\left(p+iq\right)\left(r+is\right)$ must satisfy ${\left|p+iq\right|}^2=\frac{T+M}{2}=4^{h-1}B$.
\\[10pt]
\noindent We observe that a solution $\left(z_P,z_R\right)$ to the above equations is obtained for any unit $w\in \left\{\pm 1,\pm i\right\}$ by setting
\begin{align*}
&z_P=2^{h-1}\overline{z_B}=2^{h-1}\left(a_2-2ib_1\right)w, \\
&z_R=\left(2^{h-1}z_B-\overline{z_A}\right)\overline{w}=\left[\left(-2a_1+2^{h-1}a_2\right)+i\left(2^hb_1-b_2\right)\right]\overline{w}.
\end{align*}
\noindent The additional parity and positivity conditions required by $\left(P_0\right)$ imply that $w=i$, which leads to the following result:

\begin{proposition}{:}
\label{p15}
Let $X^2+Y^2\in E^*_{Y^2}$ and $A,B>1$ such that $X^2+Y^2=AB$. Let $\left(a_1,a_2,b_1,b_2\right)$ the quadruple satisfying $\left(P'_0\right)$ with $b_1>0$. Then there exists a unique solution $\left(p,q,r,s\right)$ to $\left(P_0\right)$ given by:
\[\left\{ \begin{array}{l}
p=2^hb_1, \\ 
q=2^{h-1}a_2, \\ 
r=2^hb_1-b_2, \\ 
s=2a_1-2^{h-1}a_2. \end{array}
\right.\] 
\end{proposition}

\paragraph{\textnormal{\underbar{\textit{Proof}}}}{:}
It is sufficient to verify that the given values satisfy $\left(P_0\right)$. This can be done by direct computation: 
\begin{align*}
&p^2+q^2+r^2+s^2 \\
&=2^{2h}b^2_1+2^{2\left(h-1\right)}a^2_2+2^{2h}b^2_1-2^{h+1}b_1b_2+b^2_2+4a^2_1-2^{h+1}a_1a_2+2^{2\left(h-1\right)}a^2_2 \\
&=2^{2h-1}\left(a^2_2+4b^2_1\right)+\left(4a^2_1+b^2_2\right)-2^h\left(2a_1a_2+2b_1b_2\right) \\
&=T,
\end{align*}
\begin{align*}
&p^2+q^2-r^2-s^2 \\
&=2^{2h}b^2_1+2^{2\left(h-1\right)}a^2_2-2^{2h}b^2_1+2^{h+1}b_1b_2-b^2_2-4a^2_1+2^{h+1}a_1a_2-2^{2\left(h-1\right)}a^2_2 \\
&=2^h\left(2a_1a_2+2b_1b_2\right)-\left(4a^2_1+b^2_2\right) \\
&=M,
\end{align*}
 
\[2\left(pr-qs\right)=2^{2h-1}\left(a^2_2+4b^2_1\right)-2^h\left(2a_1a_2+2b_1b_2\right)=L\] 
\[2\left(ps+qr\right)=2^h\left(4a_1b_1-a_2b_2\right)=2^hY.\] 
Moreover, by hypothesis $p=2^hb_1\ge 0$.
\\[10pt]
\noindent Finally, since $a_2$ is necessarily odd, we have $\nu_2\left(p\right)\ge h>\nu_2\left(q\right)=h-1$.
\\[10pt]
\noindent The uniqueness of $\left(p,q,r,s\right)$ comes from the fact that from any such quadruple, we can uniquely reconstruct the solution $\left(a_1,a_2,b_1,b_2\right)$ of $\left(P'_0\right)$ with $b_1>0$. More precisely, we set:
\begin{center}
$a_2=2^{-\left(h-1\right)}q\in 2\mathbb{N}\mathrm{+}1,\ b_1=2^{-h}p\in \mathbb{N}\mathrm{,}\ a_1=\frac{1}{2}\left(q+s\right),\ b_2=p-r.$
\end{center} 

\noindent We easily obtain $B=\frac{1}{4^{h-1}} \cdot \frac{T+M}{2}=a^2_2+4b_1$ whereas:
\begin{center}
$A=T-L=p^2+q^2+r^2+s^2-2\left(pr-qs\right)=4a^2_1+b^2_2.$
\end{center} 
\noindent By \Cref{p13}, this solution is unique. $\blacksquare $

\begin{corollary}{:}
\label{c7}
Let $n=X^2+Y^2\in E^*_{Y^2}$ be composite. Then $n\in E^1_{Y^2}$ if and only if there exist $A,B>1$ such as: 
\begin{center}
$\nu_2\left(b_1\right)\ge \nu_2\left(X\right)-1$
\end{center} 
where $b_1=b_1\left(X,Y,A,B\right)$.
\end{corollary}

\paragraph{\textnormal{\underbar{\textit{Proof}}}}{:}
Let $h=\nu_2\left(X\right)$. By the definition of $E^1_{Y^2}$, $n$ belongs to $E^1_{Y^2}$ if and only if there exist $A,B>1$ such that $n=AB$ and a solution $\left(p,q,r,s\right)$ to $\left(P_0\right)$ such that $\nu_2\left(p\right)>2\left(h-1\right)$. By \Cref{p15}, $p=2^hb_1$ so this condition reduces to $\nu_2\left(b_1\right)\ge h-1$. $\blacksquare $

\begin{remark}{:}
\label{r9}
This condition always holds if $\nu_2\left(X\right)=1$, regardless of the choice of $A$ and $B$. Consequently, the corresponding elements of $E^*_{Y^2}$ do not belong to $E^2_{Y^2}$.
\end{remark}

\noindent 
\subsection{Infinitude of sets defined by identities (V1) and (V2)}

\noindent We are now in a position to establish the following propositions.

\begin{proposition}{:}
\label{p16}
$\left|E^1_{Y^2}\right|=\left|E^2_{Y^2}\right|=+\infty $.
\end{proposition}

\paragraph{\textnormal{\underbar{\textit{Proof}}}}{:}
The infinitude of $E^1_{Y^2}$ follows directly from \Cref{r9}; however, an alternative proof is provided here. Let $X^2+Y^2\in E^*_{Y^2}$ be an arbitrarily chosen composite number, $A,B>1$ such that $X^2+Y^2=AB$ and set $\left(a_1,a_2,b_1,b_2\right)=\mathcal{D}\left(X,Y,A,B\right)$. Consider the Type 2 progressions introduced in Section 3.1.2, defined by $m={\mathrm{lcm} \left(4a_1,b_2,y\right)\ }$ and, for any $k\in \mathbb{Z}$ such that $b_1+\frac{m}{4a_1}k\neq 0$:

\begin{center}
$X_k=2\left(a_1\left(a_2+\frac{m}{b_2}k\right)+\left(b_1+\frac{m}{4a_1}k\right)b_2\right),\ B_k={\left(a_2+\frac{m}{b_2}k\right)}^2+4{\left(b_1+\frac{m}{4a_1}k\right)}^2$ 
\end{center}

\noindent These progressions satisfy $\left(X_k,\ Y,A,B_k\right)\in \mathcal{F}$ and $\mathcal{D}\left(X_k,\ Y,A,B_k\right)=\left(a_1,a_2+\frac{m}{b_2}k,b_1+\frac{m}{4a_1}k,b_2\right)$.
\\[10pt]
\noindent For sufficiently large k, $n_k\coloneqq X^2_k+Y^2\in E^*_{Y^2}$ forms a strictly increasing progression of composite numbers. Let $h_k=\nu_2\left(X_k\right)$. By \Cref{c7}, $n_k\in E^1_{Y^2}$ provided that:

\begin{center}
$\nu_2\left(b_1+\frac{m}{4a_1}k\right)\ge h_k-1.$
\end{center}

\noindent Since $b_2$ and $Y$ are odd, $\frac{m}{4a_1}$ is also odd. From this,  it follows that we can make $\nu_2\left(b_1+\frac{m}{4a_1}k\right)$ arbitrarily large for infinitely many $k$. On the other hand, $a_2+\frac{m}{b_2}k$ is always odd. In particular, infinitely many terms in this progression satisfy $\nu_2\left(b_1+\frac{m}{4a_1}k\right)>\nu_2\left(a_1\right)=h_k-1$. For these terms, we have:
\begin{center}
$h_k-1=\nu_2\left(a_1\left(a_2+\frac{m}{b_2}k\right)+\left(b_1+\frac{m}{4a_1}k\right)b_2\right)=\nu_2\left(a_1\right)<\nu_2\left(b_1+\frac{m}{4a_1}k\right)$
\end{center}
which implies $n_k\in E^1_{Y^2}$. Therefore, $E^1_{Y^2}$ is infinite.
\\[10pt]
\noindent Similarly, since $\frac{b_2m}{4a_1}+\frac{a_1m}{b_2}$ is odd, we can choose $k$ such that the congruence of $a_1\left(a_2+\frac{m}{b_2}k\right)+\left(b_1+\frac{m}{4\left|a_1\right|}k\right)b_2$ modulo any power of $2$ can be be arbitrarily fixed. Thus, choosing $\alpha>\nu_2\left(a_1\right)$, there are infinitely many $k$ such that $2^\alpha|a_1\left(a_2+\frac{m}{b_2}k\right)+\left(b_1+\frac{m}{4a_1}k\right)b_2$. For any such $k$, we cannot have $2^\alpha|b_1+\frac{m}{4a_1}k$, otherwise it would imply $2^\alpha|a_1$. Hence:
\begin{center}
$\nu_2\left(b_1+\frac{m}{4a_1}k\right)<\alpha\le \nu_2\left(a_1\left(a_2+\frac{m}{b_2}k\right)+\left(b_1+\frac{m}{4a_1}k\right)b_2\right)=h_k-1$.
\end{center}
Consequently, $E^2_{Y^2}$ is also infinite.$\blacksquare $

\begin{proposition}{:}
\label{p17}
 $\left|E^1_{Y^2}\cap E^2_{Y^2}\right|=+\infty $.
\end{proposition}

\paragraph{\textnormal{\underbar{\textit{Proof}}}}{:}
Let $X^2+Y^2\in E^2_{Y^2}$ and $A,B>1$, and let $\left(a_1,a_2,b_1,b_2\right)$ be a solution to $\left(P'_0\right)$ such that $\nu_2\left(b_1\right)<h-1=\nu_2\left(a_1a_2+b_1b_2\right)$. We have $A=4a^2_1+b^2_2$. Consider the Type 3 sequences $X_k,A_k,B_k$ from Section 3.1.2, which satisfy $\mathcal{D}\left(X_k,Y,A_k,B_k\right)=\left(a_1+\frac{m}{4b_1}k,a_2+\frac{m}{b_2}k,b_1,b_2\right)$ where $m={\mathrm{lcm} \left(4b_1,b_2,y\right)\ }$. Then:

\begin{center}
$A_k=4{\left(a_1+\frac{m}{4b_1}k\right)}^2+b^2_2$
\end{center}

\noindent Let $x$ be a multiple of $Y$ and $b_2$ such that $b_1+a_2x\equiv 0\ \left[2^h\right]$ and $C=4x^2+1\in E_1\diagdown \left\{1\right\}$. By construction $C$ is odd and coprime to $Y$ and $b_2$, and thus coprime to $\frac{m}{4b_1}$. Choosing $k$ such that $a_1+\frac{m}{4b_1}k\equiv b_2x\ \left[C\right]$ and $k\equiv 0\left[2^h\right]$, we obtain $A_k\equiv 4{\left(b_2x\right)}^2+b^2_2\equiv 0\left[C\right]$. For sufficiently large $k$, we will also have $A_k>C$ (ensuring that $C$ is a proper divisor of  $A_k$).
\\[10pt]
\noindent By shifting the factor $C=4x^2+1$ from $A_k$ to $B_k$ (i.e. replacing $A_k$ with $\frac{A_k}{C}$ and $B_k$ with $B_kC$), $z_{B_k}$ becomes $z_{B_k}\left(1+2ix\right)$ and thus:
\begin{center}
 $b_1\left(X_k,Y,\frac{A_k}{C},B_kC\right)=b_1\left(X_k,Y,A_k,B_k\right)+a_2x=b_1+a_2x\equiv 0\left[2^h\right]$, 
\end{center}
\noindent whereas:
\begin{center}
$\frac{1}{2}X_k=\left(a_1+\frac{m}{4b_1}k\right)\left(a_2+\frac{m}{b_2}k\right)+b_1b_2\equiv a_1a_2+b_1b_2\ \left[2^h\right]$\end{center} 
\noindent This implies $\nu_2\left(X_k\right)-1=h-1$.
\\[10pt]
\noindent By construction, we therefore have $X^2_k+Y^2\in E^1_{Y^2}\cap E^2_{Y^2}$ for infinitely many values of $k$. Since the progression $X_k$ is increasing for sufficiently large $k$, this implies:
\begin{center}
$\left|E^1_{Y^2}\cap E^2_{Y^2}\right|=+\infty .$
\end{center}
\noindent This completes the proof. $\blacksquare $

\begin{example}{:}
$n=25=4^2+3^2=5\times 5={\left(4\times 1^2+1^2\right)}^2$. We compute $\left(a_1,a_2,b_1,b_2\right)=\left(1,1,1,1\right)$. Here, $h=2$ and $\nu_2\left(b_1\right)=0<h-1$. There is no other non-trivial factorization, but we can replace $a_1=1$ by $1+3k$ and $a_2=1$ by $1+12k$ (with $m=12$ using the notation from the previous proof) without changing the value of $Y=3$ while preserving ${\mathrm{gcd} \left(X,Y\right)\ }=1$. This yields the values:
\begin{center}
$X_k=2\left(1+3k\right)\left(1+12k\right)+2$,\ \ \ \ $A_k=4{\left(1+3k\right)}^2+1$,\ \ \ \ $B_k={\left(1+12k\right)}^2+4$.
\end{center} 
\noindent Following the previous proof, we choose $x$ multiple of 3 such that $x\equiv -1\ \left[4\right]~$: $x=3$ satisfies this.
\\[10pt]
\noindent We then set $C={4.3}^2+1=37$, which is indeed coprime to $Y$ and $b_1$. The equation $a_1+\frac{m}{4b_1}k\equiv b_2x\ \left[C\right]$ becomes $1+3k\equiv 3\ \left[37\right]$ which reduces to $k\equiv 13\ [37]$. We also require $k\equiv 0\ \left[4\right]$, hence $k\equiv 124~\left[148\right]$. We therefore set $k=124+148\kappa$.
\\[10pt]
\noindent We verify (as expected) that:
\begin{align*}
A_k &=4{\left(1+3\left(124+148\kappa\right)\right)}^2+1 \\
&=5+24\left(124+148\kappa\right)+36\left({124}^2+248\times 148\kappa+{148}^2\kappa^2\right)
\end{align*}
\noindent is indeed a strict multiple of $37$.
\\[10pt]
\noindent This gives us two factorizations:
\begin{center}
$X^2_k+Y^2=A_kB_k=\left(\frac{A_k}{37}\right).37B_k$
\end{center}
\noindent Moreover, $\nu_2\left(X_k\right)=2$ and the quadruple corresponding to the first factorization is \\
$\left(1+3k,\ 1+12k,\ 1,\ 1\right)$, meaning $X^2_k+3^2\in E^2_9$.
\\[10pt]
\noindent However,the quadruple corresponding to the second factorization is:
\begin{center}
$\left(\frac{\left(1+3k\right)-3}{37},\left(1+12k\right)-12,1+3\left(1+12k\right),\frac{1+12\left(1+3k\right)}{37}\right)$
\end{center} 
\noindent and by construction $1+3\left(1+12k\right)\equiv 0\ \left[4\right]$ thus $X^2_k+Y^2\in E^1_9$.\textit{ }$\blacksquare $
\end{example}

\begin{remark}{:}
We observe in \Cref{r9} that $E^1_{Y^2}\diagdown E^2_{Y^2}$ is infinite. Determining whether $E^2_{Y^2}\diagdown E^1_{Y^2}$ is infinite appears more difficult, but it would be a consequence of the following conjecture:
\end{remark}

\begin{conjecture}{:}
\label{h1}
$E^*_{Y^2}$ contains infinitely many numbers of the form $\pi_1\pi_2$ where $\pi_1,\pi_2$ are primes such that $\pi_1=4a^2_1+b^2_2,\ \pi_2=a^2_2+4b^2_1$ with $a_1,a_2,b_1,b_2$ odd.
\end{conjecture}

\noindent Indeed, such composite elements of $E^*_{Y^2}$ only admit one decomposition up to permutation, verifying $\nu_2\left(b_1\right)=0$ and $\nu_2\left(X\right)-1=\nu_2\left(a_1a_2+b_1b_2\right)=1$. So they are in $E^2_{Y^2}\diagdown E^1_{Y^2}$.

\noindent 
\section{Number of decompositions into sums of two squares}

\noindent In section 2.3, we established the identity $F=\Psi GH$. In this section, we propose to count the number of ways to express a number as sums of two squares. To this end, we introduce:
\[\phi\left(n\right)=\left|\left\{\left(X,Y\right)\in {\mathbb{N}}^2|X\ge Y\ \mathrm{et\ }n=X^2+Y^2\right\}\right|\] 
For $n\notin F$, we have $\phi\left(n\right)=0$.

\noindent 
\subsection{Link with factorizations of $\boldsymbol{n}$}

\noindent Let $n\in \Psi$. We have shown that there is a correspondence between the expressions $n=X^2+Y^2$ and the factorizations $n=AB$ via the equations $X+iY=z_Az_B,\ A={\left|z_A\right|}^2,\ B={\left|z_B\right|}^2$, in the case where $X$ and $Y$ are coprime.
\\[10pt]
\noindent This correspondence can be reversed. To any factorization $n=AB$ we can associate $z_A,z_B$ constructed as the product of the irreducible divisors of $A$ in $Z^+_{\mathbb{P}}$ and $B$ in $Z^-_{\mathbb{P}}$, respectively. This ensures ${\left|z_A\right|}^2=A,{\left|z_B\right|}^2=B$, since they are both integers and have no divisors in $\left\{2\right\}\cup {\mathcal{P}}_{3,4}$.
\\[10pt]
\noindent We then form $X+iY=z_Az_B$. By construction $n=X^2+Y^2$.

\begin{proposition}{:}
\label{p18}
For all $n\in \Psi$, $\phi\left(n\right)=\left|\left\{\left(A,B\right)\in {{\mathbb{N}}^*}^2\mathrel{\left|\vphantom{\left(A,B\right)\in {{\mathbb{N}}^*}^2 A\ge B\ \mathrm{et}\ n=AB}\right.\kern-\nulldelimiterspace}A\ge B\ \mathrm{and}\ n=AB\right\}\right|$.
\end{proposition}

\paragraph{\textnormal{\underbar{\textit{Proof}}}}{:}
The construction described above defines a function
\begin{center}
$f:\left\{\left(A,B\right)\in {{\mathbb{N}}^*}^2\mathrel{\left|\vphantom{\left(A,B\right)\in {{\mathbb{N}}^*}^2 A\ge B\ \mathrm{et\ }AB=n}\right.\kern-\nulldelimiterspace}A\ge B\ \mathrm{et\ }AB=n\right\}\to \left\{\left(X,Y\right)\in \mathbb{N}\mathrel{\left|\vphantom{\left(X,Y\right)\in \mathbb{N} X\ge Y\ \mathrm{et\ }X^2+Y^2=n}\right.\kern-\nulldelimiterspace}X\ge Y\ \mathrm{et\ }X^2+Y^2=n\right\}.$
\end{center}

\noindent We want to show that $f$ is bijective.
\\[10pt]
\noindent We begin by proving injectivity. To do so, we observe that $f\left(A,B\right)=f\left(A',B'\right)$ implies (up to a unit multiplication) $z_Az_B=z_{A'}z_{B'}$ or $z_Az_B=\overline{z_{A'}}\overline{z_{B'}}$. By identifying the irreducible factors of the positive and negative imaginary part, $z_A=z_{A'}$ and $z_B=z_{B'}$ (or vice versa), from which it follows that $\left(A',B'\right)=\left(A,B\right)$ or $\left(B,A\right)$. However, the assumption $A\ge B$ forces $\left(A',B'\right)=\left(A,B\right)$.
\\[10pt]
\noindent We now show surjectivity. Let $X,Y$ be such that $n=X^2+Y^2$ and decompose $X+iY=z_Az_B$ (up to a unit) by setting $z_A$ the product of its irreducible factors of $Z^+_{\mathbb{P}}$ and $z_B$ the product of its irreducible factors in $Z^-_{\mathbb{P}}$. Up to conjugation of $X+iY$ to swap $z_A$ and $z_B$ and ensure that ${\left|z_A\right|}^2\ge \ {\left|z_B\right|}^2$, we obtain the preimage:
\[f\left({\left|z_A\right|}^2,{\left|z_B\right|}^2\right)=n.\] 
This completes the proof. $\blacksquare $
\\[10pt]
\noindent We now generalize this result to any integer in $F$:

\begin{theorem}{:}
\label{t3}
Let $n=\psi gh\in F$ with $\left(\psi,\ g,h\right)\in \Psi \times G \times H$. Then:
\begin{center}
$\phi\left(n\right)=\left|\left\{\left(A,B\right)\in {{\mathbb{N}}^*}^2\mathrel{\left|\vphantom{\left(A,B\right)\in {{\mathbb{N}}^*}^2 A\ge B\mathrm{\ et\ }\psi=AB}\right.\kern-\nulldelimiterspace}A\ge B\mathrm{\ et\ }\psi=AB\right\}\right|$.
\end{center}

\end{theorem}

\paragraph{\textnormal{\underbar{\textit{Proof}}}}{:}
Any decomposition $n=X^2+Y^2$ is equivalent to the factorization $n=\left(X+iY\right)\left(X-iY\right)$. However, the irreducible factors of $X+iY$ in $\left\{1+i\right\}\cup {\mathcal{P}}_{3,4}$ are necessarily identical to those of $X-iY$, so there is no choice involved in their distribution. They correspond to $gh$ in the decomposition $n=\psi gh$. The problem then reduces to the case of the previous proposition. $\blacksquare $

\begin{remark}{:}
This identity can also be generalized in most cases to:
\[\left|\left\{\left(X,Y\right)\in {\mathbb{N}}^2\mathrel{\left|\vphantom{\left(X,Y\right)\in {\mathbb{N}}^2 n=X^2+Y^2}\right.\kern-\nulldelimiterspace}n=X^2+Y^2\right\}\right|=\left|\left\{\left(A,B\right)\in {{\mathbb{N}}^*}^2\mathrel{\left|\vphantom{\left(A,B\right)\in {{\mathbb{N}}^*}^2 \psi=AB}\right.\kern-\nulldelimiterspace}\psi=AB\right\}\right|.\] 
Indeed, it basically amounts to counting each pair $\left(X,Y\right)$ such that $X>Y$ and each pair $\left(A,B\right)$ such that $A>B$ twice. There is, however, one specific case to distinguish. If $\psi$ is a perfect square and $n$ is not twice a perfect square, then the divisors of $\psi$ contain an element of the form $\left(A,A\right)$ which is not duplicated, whereas its decompositions $n=X^2+Y^2$ always satisfy $X\neq Y$ because $n$ is not twice a perfect square. In this case, we have:
\[\left|\left\{\left(X,Y\right)\in {\mathbb{N}}^2\mathrel{\left|\vphantom{\left(X,Y\right)\in {\mathbb{N}}^2 n=X^2+Y^2}\right.\kern-\nulldelimiterspace} n=X^2+Y^2\right\}\right|=\left|\left\{\left(A,B\right)\in {{\mathbb{N}}^*}^2\mathrel{\left|\vphantom{\left(A,B\right)\in {{\mathbb{N}}^*}^2 \psi =AB}\right.\kern-\nulldelimiterspace}\psi =AB\right\}\right|+1.\]
\end{remark}

\begin{example}{:}
$25=0^2+5^2=5^2+0^2=3^2+4^2=4^2+3^2=1.25=25.1=5.5$
\end{example}

\begin{remark}{:}
\label{r12}
Given the prime factorization $\psi=\prod_{p\in {\mathcal{P}}_{1,4}}{p^{\nu_p\left(n\right)}}$ we can easily determine:

\begin{center}
$\left|\left\{\left(A,B\right)\in {{\mathbb{N}}^*}^2\mathrel{\left|\vphantom{\left(A,B\right)\in {{\mathbb{N}}^*}^2 \psi=AB}\right.\kern-\nulldelimiterspace}\psi=AB\right\}\right|=d\left(\psi\right)=\prod_{p\in {\mathcal{P}}_{1,4}}{\left(\nu_p\left(n\right)+1\right)}$
\end{center}
\noindent where $d\left(\psi\right)$ denotes the number of divisors of $\psi$ in ${\mathbb{N}}^*$.
\end{remark}

\noindent Another formula commonly encountered  \cite{b11} \cite{b12} is:
 
\begin{center}
$\left|\left\{\left(X,Y\right)\in {\mathbb{Z}}^2\mathrel{\left|\vphantom{\left(X,Y\right)\in {\mathbb{Z}}^2 n=X^2+Y^2}\right.\kern-\nulldelimiterspace}n=X^2+Y^2\right\}\right|=4\left(d_1\left(n\right)-d_3\left(n\right)\right)$
\end{center}
\noindent where $d_k\left(n\right)=\left|\mathcal{D}\left(n\right)\cap \left(4\mathbb{Z}\mathrm{+}k\right)\right|$. Let us prove this using our results.
\\[10pt]
\noindent When $\psi$ is not a perfect square or if $n$ is twice a perfect square, the expression $n=X^2+Y^2$ implies $X,Y\neq 0$ and we simply have: 

\begin{center}
$\left|\left\{\left(X,Y\right)\in {\mathbb{Z}}^2\mathrel{\left|\vphantom{\left(X,Y\right)\in {\mathbb{Z}}^2 n=X^2+Y^2}\right.\kern-\nulldelimiterspace}n=X^2+Y^2\right\}\right|=4\left|\left\{\left(X,Y\right)\in {\mathbb{N}}^2\mathrel{\left|\vphantom{\left(X,Y\right)\in {\mathbb{N}}^2 n=X^2+Y^2}\right.\kern-\nulldelimiterspace}n=X^2+Y^2\right\}\right|.$
\end{center}

\noindent Otherwise, we must isolate the cases where $X$ or $Y$ is zero. We then have: 

\begin{center}
$\left|\left\{\left(X,Y\right)\in {\mathbb{Z}}^2\mathrel{\left|\vphantom{\left(X,Y\right)\in {\mathbb{Z}}^2 n=X^2+Y^2}\right.\kern-\nulldelimiterspace}n=X^2+Y^2\right\}\right|=4\left(\left|\left\{\left(X,Y\right)\in {\mathbb{N}}^2\mathrel{\left|\vphantom{\left(X,Y\right)\in {\mathbb{N}}^2 n=X^2+Y^2}\right.\kern-\nulldelimiterspace}n=X^2+Y^2\right\}\right|-1\right).$
\end{center}

\noindent In both case, \Cref{r12} ensures that $\left|\left\{\left(X,Y\right)\in {\mathbb{Z}}^2\mathrel{\left|\vphantom{\left(X,Y\right)\in {\mathbb{Z}}^2 n=X^2+Y^2}\right.\kern-\nulldelimiterspace}n=X^2+Y^2\right\}\right|=4d\left(\psi\right)$. Additionally:
\[d_1\left(n\right)=d\left(\psi\right).d_1\left(g\right),\ \ d_3\left(n\right)=d\left(\psi\right).d_3\left(g\right).\] 

\noindent We will show that, for any $g\in G$, $d_1\left(g\right)=d_3\left(g\right)+1$ by complete induction on $g$. Clearly this holds for $g=1$. Let $g\ge 2$ and assume the property holds for any element $\gamma \in G$ such that $\gamma <g$ (the induction hypothesis). $g$ must have a prime divisor $p\in {\mathcal{P}}_{3,4}$, let $g'$ such that $g=g'p^{{\nu }_p\left(g\right)}$ we can compute:

\begin{align*}
d_1\left(g\right) &=d_1\left(g'\right)d_1\left(p^{\nu_p\left(g\right)}\right)+d_3\left(g'\right)d_3\left(p^{\nu_p\left(g\right)}\right),\ \ d_3\left(g\right) \\
&=d_1\left(g'\right)d_3\left(p^{\nu_p\left(g\right)}\right)+d_3\left(g'\right)d_1\left(p^{\nu_p\left(g\right)}\right).
\end{align*}
 
\noindent We easily show, since g is a perfect square, that  $d_1\left(p^{\nu_p\left(g\right)}\right)=\frac{\nu_p\left(g\right)}{2}+1$ and $d_3\left(p^{\nu_p\left(g\right)}\right)=\frac{\nu_p\left(g\right)}{2}$. Furthermore, $g^{'}$ is a perfect square smaller than $g$ so $d_1 (g^{'})=d_3 (g^{'})+1$. Thus:

\begin{align*}
d_1\left(g\right) &=d_1\left(g'\right)\left(\frac{\nu_p\left(g\right)}{2}+1\right)+\left(d_1\left(g'\right)-1\right)\frac{\nu_p\left(g\right)}{2} \\
&=d_1\left(g'\right)\frac{\nu_p\left(g\right)}{2}+\left(d_1\left(g'\right)-1\right)\left(\frac{\nu_p\left(g\right)}{2}+1\right)+1=d_3\left(g\right)+1.
\end{align*}

\noindent We can therefore conclude that:
\begin{center}
$\left|\left\{\left(X,Y\right)\in {\mathbb{Z}}^2\mathrel{\left|\vphantom{\left(X,Y\right)\in {\mathbb{Z}}^2 n=X^2+Y^2}\right.\kern-\nulldelimiterspace}n=X^2+Y^2\right\}\right|=4d\left(\psi\right)=4\left(d_1\left(n\right)-d_3\left(n\right)\right).$
\end{center}

\begin{corollary}{:}
We have :
\begin{center}
$\phi\left(n\right)=\left\lceil \frac{1}{2}\prod_{p\in {\mathcal{P}}_{1,4}}{\left(\nu_p\left(n\right)+1\right)}\right\rceil .$
\end{center}

\noindent In particular, if we set $S=\sum_{p\in {\mathcal{P}}_{1,4}}{\nu_p\left(n\right)}$, $\phi\left(n\right)$ is bounded as follows: 
\begin{center}
$\frac{S+1}{2}\le \phi\left(n\right)\le 2^{S-1}$
\end{center}
\end{corollary}

\paragraph{\textnormal{\underbar{\textit{Proof}}}}{:}
The identity follows from the above and from $d\left(\psi\right)=\prod_{p\in {\mathcal{P}}_{1,4}}{\left(\nu_p\left(n\right)+1\right)}$.
\\[10pt]
\noindent Furthermore, expanding the product shows that $\prod_{p\in {\mathcal{P}}_{1,4}}{\left(\nu_p\left(n\right)+1\right)}\ge 1+S$ since the right-hand side corresponds to a subset of terms in the expansion, and all missing terms are positive.
\\[10pt]
\noindent Finally, since $\left(1+k\right)\le 2^k$ for all $k\ge 0$, the upper bound is straightforward to verify. $\blacksquare $

\begin{remark}{:}
These two bounds are sharp: the lower bound is achieved for odd powers of a single prime number, and the upper bound is achieved for a product of pairwise distinct primes.
\end{remark}

\noindent 
\subsection{Algorithm for decomposing integers into sums of two squares using a sieve}

\noindent We now know that the number of representations of an integer as a sum of two squares is obtained by determining the prime factors congruent to 1 modulo 4 of $n$ along with their multiplicities. Such decompositions yield four sequences that allow for the discovery of factorizations of other integers that are sums of two squares.
\\[10pt]
\noindent In Section 4.1 of \cite{a20}, a sieve algorithm is given that determines the prime numbers in $E_c$ for $c\in {\mathbb{N}}^*$, and more generally the complete factorization of the elements of $E_c$. The results of this article therefore make it straightforward to extend this algorithm so that it also returns $\phi \left(n\right)$ when $c=Y^2$ is a perfect square.
\\[10pt]
\noindent Here, we present a more comprehensive version to determine, for all $n\in E^*_{Y^2}\cap \llbracket 1, N \rrbracket$, not only $\phi \left(n\right)$ but also all the corresponding factorizations and representations as a sum of two squares. The steps presented below complement the aforementioned algorithm. This extension is based on the following result:

\begin{proposition}{:}
\label{p19}
Let $n=X^2+Y^2\in E^*_{Y^2}$ be composite. The mapping: 
\begin{center}
$\Gamma_Y:\left(A,B\right)\mapsto \left(X_{A,B},Y_{A,B}\right)\coloneqq \left(2\left(a_1a_2-b_1b_2\right),4a_1b_1+a_2b_2\right)$
\end{center}
\noindent defines a bijection from the set of factorizations $n=AB$ to the representations of $n$ as sums of two squares.
\end{proposition}

\noindent \textit{\underbar{Proof}}: Writing $z_A=2a_1-ib_2,\ z_B=a_2+2ib_1$ as in Section 3, so that $A={\left|z_A\right|}^2,\ B={\left|z_B\right|}^2$ and $X+iY=z_Az_B$, it is equivalent to set $X_{A,B}+iY_{A,B}=\overline{z_A}z_B$. Indeed, on the one hand, for any decomposition $n={\mathcal{X}}^2+{\mathcal{Y}}^2$, let $z_B$ be the greatest common divisor of $X+iY$ and $\mathcal{X}+i\mathcal{Y}$, and let $z_A=\frac{X+iY}{z_B}$. By construction since $z_A$ and $\frac{\mathcal{X}+i\mathcal{Y}}{z_B}$ are coprime and have the same modulus, they are complex conjugates. That is to say, $\mathcal{X}+i\mathcal{Y}=\overline{z_A}z_B$. This is equivalent to saying that $\Gamma\left(A\coloneqq {\left|z_A\right|}^2,{B\coloneqq \left|z_B\right|}^2\right)=\left(\mathcal{X},\mathcal{Y}\right)$, which proves $\Gamma_Y$ is surjective.
\\[10pt]
\noindent Now, if there were another preimage $\left(A',B'\right)$ of $\left(\mathcal{X},\mathcal{Y}\right)$, we would have:
\[X+iY=z_Az_B=z_{A'}z_{B'},\ \ \ \ \mathcal{X}+i\mathcal{Y}=\overline{z_A}z_B=\overline{z_{A'}}z_{B'}\] 
Up to dividing these equations by their respective greatest common divisor, we can assume $z_A$ and $z_{A'}$ (respectively $z_B$ and $z'_B$) are coprime, which implies $z_{B'}=\pm iz_A$ and $z_{A'}=\mp iz_B$, and thus $n={\mathcal{Y}}^2$. However, since $z_Az_B$ does not contain conjugated irreducible factors, the only possibility is $n=1$, which proves that $\left(A,B\right)=\left(A',B'\right)$. Thus $\Gamma_Y$ is also injective. $\blacksquare $
\\[10pt]
\noindent The additional steps of the algorithm are:

\begin{enumerate}[label=-]
\item  Initialize a correspondence table between $\left(A,B\right)$ and $\left(X_{A,B},Y_{A,B}\right)$ for any value of $X$ coprime to $Y$.

\item  For any value of $X$ coprime to $Y$, deduce all factorizations $n=AB$ from the decomposition of $n=X^2+Y^2$ (which is what the original sieve already does).

\item  For the pairs $\left(A,B\right)$ obtained above and \underbar{that are not already in the correspondence table}, construct the quadruple $\left(a_1,a_2,b_1,b_2\right)$ and from it deduce the mapping $\left(X_{A,B},Y_{A,B}\right)$  as well as the mappings of the pairs  $\left(A_k,B_k\right)$ for the four types of sequences of the form $\left(X^2_k+Y^2\right)$ defined in Section 3.
\end{enumerate}

\noindent An illustration of a step of this extended algorithm is provided in Appendix A.

\noindent 
\section{Conclusion }

\noindent We have described the arithmetic structure of the non-trivial integers of the form $n\mathrm{=}X^{\mathrm{2}}\mathrm{+}Y^{\mathrm{2}}$ (for a fixed odd integer $Y$), and shown how their non-trivial factorizations $n\mathrm{=}AB$ can be interpreted and classified via Diophantus’ identity and the Gaussian integers. This explicitly links them to a unique quadruple $\left(a_{\mathrm{1}},a_{\mathrm{2}},b_{\mathrm{1}},b_{\mathrm{2}}\right)$, which also gives rise to a second representation of $n$ as a sum of two squares. Families of affine or quadratic progressions constructed from this quadruple generate infinite subsets of integers $n$ that preserve a similar arithmetic structure.
\\[10pt]
\noindent On a practical level, these identities and the associated progressions can be incorporated into a sieve algorithm to determine recursively all representations of the elements of $E^{\mathrm{*}}_{Y^{\mathrm{2}}}$ as sums of two squares. In this manner, all the theoretical results of the article can be implemented in an elementary and practical way.
\\[10pt]
\noindent \underbar{Acknowledgements} : We thank Fran\c{c}ois-Xavier VILLEMIN for his attentive comments and valuable suggestions.

\noindent 
\addcontentsline{toc}{section}{REFERENCES}

\bibliography{biblioarticle9.bib}

\noindent 

\noindent 
\section*{APPENDIX A: One Step of the Sieve}
\addcontentsline{toc}{section}{APPENDIX A: One Step of the Sieve}

\noindent Below, we present the steps for calculating the sieve extension introduced in Section 4.2, illustrated for a single iteration.
\\[10pt]
\noindent We therefore fix $n=65$,  $\left(X,Y\right)=\left(8,1\right)$ and the decomposition $n=5\cdot 13~$:

\begin{enumerate}
\item  We factor $X+iY=8+i=\left(2-i\right)\left(3+2i\right)$ from which we deduce the quadruple \\ $\left(a_1,a_2,b_1,b_2\right)=\left(1,3,1,1\right)$.

\item  The mapping $\Gamma_Y$ yields $\left(X_{5,13},Y_{5,13}\right)=\left(4,7\right)$ so $65=16+49$.

\item  We construct the first terms of the four sequences (for brevity, we construct only the terms for $k=\pm 1$). The product of the prime numbers dividing $Y=1$ is $y=1$. We can then consider the following cases:

\begin{enumerate}[label=$\bullet$]
\item  Type 1 progression: \\
$m={\mathrm{lcm} \left(a_2,4b_1,y\right)\ }=12$, $X_k=2\left(\left(a_1+\frac{m}{4b_1}k\right)a_2+b_1\left(b_2+\frac{m}{a_2}k\right)\right)=26k+8$.
For any $k$, we can add the correspondence $\left(A_k=4{\left(1+3k\right)}^2+{\left(1+4k\right)}^2,B=13\right)$ which is given by: 
\begin{center}
$\Gamma_{\mathrm{1}}\left(A_k,B\right)=\left(\left|4+10k\right|,\left|7+24k\right|\right).$
\end{center}
For $k=1$, we introduce the correspondence $\Gamma_1\left(89,13\right)=\left(14,31\right)$. For $k=-1$ we introduce $\Gamma_1\left(25,13\right)=\left(6,17\right)$.

\item  Type 2 progression: \\
$m={\mathrm{lcm} \left(4a_1,b_2,y\right)\ }=4$, $X_k=8+10k$. For any $k$, we add the correspondence $\left(A=5,B_k={\left(3+4k\right)}^2+4{\left(1+k\right)}^2\right)$ which is given by: 
\begin{center}
$\Gamma_1\left(A,B_k\right)=\left(\left|4+6k\right|,\left|7+8k\right|\right).$ 
\end{center}
For $k=\pm 1$ we introduce the correspondences $\Gamma_1\left(5,65\right)=\left(10,15\right)$ and $\Gamma_1\left(5,1\right)=\left(2,1\right)$ (the latter having already been identified by the sieve and corresponding to a prime number).

\item  Type 3 progression: \\
$m={\mathrm{lcm} \left(4b_1,b_2,y\right)\ }=4$, $X_k=8+14k+8k^2$. For any $k$, we add the correspondence $\left(A_k=4{\left(1+k\right)}^2+1\mathrm{,\ }B_k={\left(3+4k\right)}^2+4\right)$ which is given by:
\begin{center}
$\Gamma_1\left(A_k,B_k\right)=\left(\left|4+14k+8k^2\right|,\left|7+8k\right|\right).$ 
\end{center}
For $k=\pm 1$ we introduce the correspondences $\Gamma_1\left(17,53\right)=\left(26,15\right)$ and $\Gamma_1\left(1,5\right)=\left(2,1\right)$ (the second latter having already been identified by the sieve and corresponding to a prime number).

\item  Type 4 progression: \\
$m={\mathrm{lcm} \left(4a_1,a_2,y\right)\ }=12$, $X_k=8+14k+24k^2$. For any $k$, we add the correspondence $\left(A_k=4+{\left(1+4k\right)}^2,B_k=9+4{\left(1+3k\right)}^2\right)$ which is given by:
\begin{center}
$\Gamma_1\left(A_k,B_k\right)=\left(\left|4-14k-24k^2\right|,\left|7+24k\right|\right).$ 
\end{center}
For $k=\pm 1$ we introduce the correspondences $\Gamma_1\left(29,73\right)=\left(34,31\right)$ and $\Gamma_1\left(13,25\right)=\left(6,17\right)$ (the latter having already been found via the Type 1 progression).
\end{enumerate}
\end{enumerate}

\noindent The correspondences discovered by the sieve using only the indices $k=\pm 1$ are summarized below:

\begin{center}
\begin{tabular}{|p{1in}|p{1in}|p{1in}|p{1in}|p{1in}|}
 \hline 
$n$ & $A$ & $B$ & $X_{A,B}$ & $Y_{A,B}$ \\ \hline 
65 & 5 & 13 & $4$ & $7$ \\ \hline 
1157 & 89 & 13 & 14 & 31 \\ \hline 
325 & 25 & 13 & 6 & 17 \\ \hline 
325 & 5 & 65 & 10 & 15 \\ \hline 
901 & 17 & 53 & 26 & 15 \\ \hline 
2117 & 29 & 73 & 34 & 31 \\ \hline 
\end{tabular}
\end{center}

\noindent The "yield" of the Type 2 progression (which fixes the smaller of the two factors and varies X linearly) appears to be particularly high: all terms for $-4\le k\le 3$ yield a value of $n$ lower than that produced by Type 4 progression for $k=1$.

\noindent 

\end{document}